\documentclass[10pt,leqno]{amsart}

   \usepackage[centertags]{amsmath}
   \usepackage{amsfonts}
   \usepackage{amsmath}
   \usepackage{amssymb}
   \usepackage{amsthm}
   \usepackage{newlfont}

\allowdisplaybreaks[3]

\theoremstyle{plain}
\newtheorem{Theo}{Theorem}

\newtheorem{Lem}[Theo]{Lemma}
\newtheorem{Def}{Definition}

\usepackage{xcolor}

\theoremstyle{remark}
\newtheorem{Rem}{Remark}

\newcommand\A{{\mathcal A}}
\newcommand\B{{\mathcal B}}
\newcommand\D{{\mathcal D}}
\newcommand\W{{\mathcal W}}
\newcommand{\F}{\mathcal{F}}
\newcommand\Pp{{\mathcal P}}
\newcommand{\Aa}{\mathbb{A}}
\newcommand{\Bb}{\mathbb{B}}
\newcommand{\PP}{\mathbb{P}}
\newcommand{\RR}{\mathbb{R}}

\newcommand{\NN}{\mathbb{N}}

\numberwithin{equation}{section}

\newcommand\Rr{{\mathcal R}}

\newcommand\supp{\operatorname{Supp}}

\newcommand\comp{\operatorname{\hspace{-2pt}{}^{{}_{{}_\circ}}\hspace{-2pt}}}

\catcode`,\active

\catcode`\,12

\title[Exceptional matrix polynomials]{Using quasi-Darboux transformations to construct exceptional matrix polynomials}
\author[I. Bono Parisi]{Ignacio Bono Parisi}
\address{%
        CIEM-FaMAF\\ Universidad Nacional de C\'or\-do\-ba\\
CP 5000, C\'or\-do\-ba,  Argentina}
\email{ignacio.bono@unc.edu.ar}
\author[A.~J.~Dur\'an]{Antonio J. Dur\'an}
\address{%
        Departamento de An\'alisis Matem\'atico and IMUS,
        Universidad de Sevilla,
        Sevilla, Spain}
\email{duran@us.es}
\author[I. Zurri\'an]{Ignacio N. Zurri\'an}
\address{%
        Departamento de Matemática Aplicada II and IMUS,
        Universidad de Sevilla,
        Sevilla, Spain}
\email{ignacio.zurrian@fulbrightmail.org}
   \date{}

\thanks{The research of the authors was partially supported by SeCyT-UNC , PIP 33620230100819CB, CONICET, PIP 1220150100356 and by AUIP and Junta de Andalucía PMA2-2023-056-14 (I. Bono Parisi),
 PID2021-124332NB-C21
(Minis\-te\-rio de Cien\-cia e Inno\-va\-ci\'on and Feder Funds (European Union)), and
FQM-262 (Jun\-ta de Anda\-lu\-c\'ia) (A.J. Dur\'an and I. Zurri\'an) and  Grant AUIP PMA2-2024-120-14 (I. Zurri\'an).}



\begin{document}
\begin{abstract}
We introduce a couple of methods to construct exceptional matrix polynomials. One of them uses what we have called quasi-Darboux transformations. This seems to be a more powerful method to deal with the non-commutativity problems that appear when matrix-valued polynomials are considered. The other method does not use any transformation of Darboux type. Using both methods, we construct
a collection of five illustrative examples that show how powerful our two methods are. The examples include exceptional matrix polynomials of Hermite, Laguerre, and Gegenbauer type, as well as an example with a weight matrix having a Dirac delta.
\end{abstract}

   \maketitle

\section{Introduction}

The theory of scalar valued orthogonal polynomials has played an important role
in many areas of mathematics. The best known examples, the classical families of Hermite, Laguerre and Jacobi, happen to
be eigenfunctions of  second order differential operators in the spectral parameter and they play a crucial role in mathematical physics.

Back in 1949 M.G. Krein considered a matrix valued version of the general theory.
The question of the existence of matrix valued weights such that the orthogonal polynomials
would be eigenfunctions of some second order differential operator with matrix coefficients and a matrix eigenvalue was raised by one of us in
\cite{Dur6}. Interest in this question was rekindled by the appearance of \cite{DuGr1,GPT1,GPT2}, which was followed by a flurry of activity. We call classical matrix orthogonal polynomials to a sequence of orthogonal matrix polynomials $(P_n)_n$ with respect to a weight matrix (see Section \ref{pre} for detailed definitions) that are also eigenfunctions of a second order differential operator of the form
\begin{equation}\label{dsop}
D(y)=y''F_2+y'F_1+yF_0,
\end{equation}
where $F_j$ are matrix polynomials of degree at most $j$, $0\le j\le 2$.  The weight matrix, with respect to which a sequence of classical matrix polynomials is orthogonal, is called a classical matrix weight.

Almost at the same time, exceptional orthogonal polynomials appeared (see, for instance
\cite{DEK,Durx1,Durx2,Durx3,Durx4,GUKM1,GUKM2}, where the adjective \textrm{exceptional} for this topic was introduced, \cite{GFGM,OS0,OS,Qu}, and the references therein), and there has also
been a remarkable activity around them. Exceptional orthogonal polynomials are complete orthogonal polynomial systems with respect to
a positive measure, which in addition are eigenfunctions of a second order differential operator.
They extend the classical families of Hermite, Laguerre, and Jacobi. The most apparent difference
between classical orthogonal polynomials and exceptional orthogonal polynomials is that the exceptional families have a finite number of gaps in their degrees, in the sense that not all degrees
are present in the sequence of polynomials (appearing then rational functions in the coefficients of the differential operator), although they
form a complete orthonormal set of the underlying Hilbert space defined by the orthogonalizing positive measure.

The exceptional polynomials fall into the concept of bispectral function, where one has $P_n(t)$ that satisfies a second order differential equation on $t$ and a difference equation (of order two or higher) on $n$.
It is worth mentioning that in \cite{Reach}, the author gives a method to produce bispectral functions of two variables (one discrete and one continuous) motivated by \cite{DG}. In particular, this method allows one, by applying Darboux transformations to classical orthogonal polynomials, to generate bispectral polynomials satisfying higher order difference equations (on $n$).

In \cite{CGYZ}, with the aim of establishing an intrinsic relationship between discrete-continuous bispectral functions and the prolate spheroidal phenomenon, the authors present a general tool to build matrix-valued bispectral functions satisfying difference equations of arbitrary order by applying a noncommutative version of the Darboux transformation.
Many of these examples arose from classical discrete-continuous bispectral functions (and polynomials).

In \cite{KMR}, the authors 
considered a matrix-valued setup for exceptional polynomials and constructed an example of Laguerre exceptional polynomials by using the Darboux transformation (see Definition \ref{dxmp} for the definition of exceptional matrix polynomials).

The Darboux transformation (actually credited by Darboux himself to Moutard) has been reinvented
several times in the context of differential operators.
We should also note that the same term
Darboux transformation has been used in different fields with different meanings, sometimes causing
a certain amount of confusion (it is also known as the factorization method, and in the theory of
integrable systems, yet other transformations receive the same name). In our context, we mean the
factorization of a second order differential operator and permutation of the two factors. More precisely.

\begin{Def}\label{dxt}
Given a system $(D,(P_n)_n)$ formed by a second order differential operator $D$ and a sequence $(P_n)_n$ of eigenfunctions for $D$, $D(P_n)=\Gamma_nP_n$, by a Darboux transformation of the system $(D,(P_n)_n)$ we
mean the following. For a scalar $\Psi\in \RR$, we factorize $D-\Psi $ as the product of two first order differential operators $D(y)=\B(\A(y))+y\Psi$. We then produce a new system consisting in the operator $\D$, obtained by reversing the order of the factors,
$\D(y)=\A(\B(y))+y\Psi$, and the sequence of eigenfunctions $\Pp_{n} =\A(P_n)$, satisfying $\D(\Pp_n)=\Gamma_n\Pp_{n}$.
We say that the system $(\D,(\Pp_{n})_n)$ has been obtained by applying a (one-step) Darboux transformation with parameter $\Psi$ to
the system $(D,(P_n)_n)$.
\end{Def}

In the previous definition, we can consider scalar or matrix valued operators and functions. In the second case, the differential operator acts on the right-hand side.

The Darboux transformation is one of the two methods to construct (scalar valued) exceptional polynomials; the other one being the dualization of Krall discrete polynomials and then passing to the limit (\cite{Durx1,Durx2,Durx3,Durx4}). The relationship between exceptional polynomials and the Darboux transformation was shown by C. Quesne (2008), \cite{Qu}, shortly after the first examples of exceptional polynomials
were introduced by G\'omez-Ullate, Kamran, and Milson in 2007. The Darboux transformation is applied to the second-order differential operators associated with the classical families of orthogonal polynomials.

In the scalar context in order to construct the first-order differential operator $\A$ in Definition \ref{dxt} one can choose what it is called a seed function $P$, that is, an eigenfunction of the operator $D$
 and define
$$
\A(y)=-yP'+y'P.
$$
If $P$ is a polynomial then $\A(P)$ is also a polynomial. We point out that the seed function has to satisfy other additional assumptions: for instance, $P'/P$ has to be a rational function without poles in the support of the underlaying weight function.

In the matrix setting, non-commutativity problems appear. To avoid them when applying the Darboux transformation in \cite{KMR}, the authors consider seed functions $P$ associated to eigenfunctions of $D$ with scalar eigenvalues and then define
$$
\A(y)=-yP^{-1}P'+y'.
$$
This has the inconvenient that second-order differential operators $D$ associated to classical matrix weights
having eigenfunctions $P$ with scalar eigenvalues (and satisfying the other assumptions) are extremely rare.

The main purpose of this paper is to introduce what we have called quasi-Darboux transformations to construct exceptional matrix polynomials. As we will show, this seems to be a more powerful method to deal with the non-commutativity problems which appear when matrix-valued polynomials are considered. Besides that method, we introduce another technique to construct exceptional matrix polynomials that does not use any transformation of Darboux type.

The content of this paper is as follows.

In Section \ref{pre} we include some preliminaries on matrix orthogonality and the key concept of symmetry for differential operators of the form
\begin{equation}\label{diop}
D(y)=\sum_{j=0}^k y^{(j)}F_j
\end{equation}
where $F_j$ are matrix polynomials of degree at most $j$ (so that $\deg D(P)\le P$). Orthogonality is with respect to a weight matrix (a positive definite matrix of measure $W$ supported in the real line having finite moments of any order and infinite support). Because of the reasons we will explain below, we also consider \textit{signed} weight matrices, i.e., a Hermitian matrix of measure $W$ having moments of any order and infinite support. We note here that when dealing with classical families of orthogonal matrix polynomials, we find two interesting phenomena which do not appear in the scalar case:
\begin{enumerate}
\item The elements of a family of orthogonal
matrix polynomials $(P_n)_n$ can be common eigenfunctions of several linearly independent
second-order differential operators (while in the scalar case, the second-order differential operator is unique up to multiplicative and additive constants). As a consequence of this phenomenon, the algebra of differential operators associated with a fixed classical weight matrix $W$ has received a lot of attention (see
\cite{CaGr1,CaGr2,DuDI1,GrTi,Tir2,BP1,BP2,BP3,Z}).
\item  There are examples of second-order differential operators $D$ (as in (\ref{diop}) for $k=2$) for
which there exist two weight matrices $W_1$ and $W_2$, $W_1\not = \alpha W_2$ for any $\alpha > 0$, such that $D$ is
symmetric with respect to any of the weight matrices $\gamma W_1 + \xi W_2$, $\gamma,\xi\ge 0$ (see \cite{DuDI2}).
\end{enumerate}

In Section \ref{exc} we consider exceptional matrix polynomials. As in the scalar case, exceptional matrix polynomials are complete orthogonal polynomial systems with respect to a weight matrix, each polynomial with nonsingular leading coefficient, which in addition are eigenfunctions of a second order differential operator, although not all degrees
are present in the sequence of polynomials, and then, rational functions can appear in the coefficients of the differential operator. This suggests considering the concept of symmetric differential operators $\D$ with respect to the pair $(\W,\Aa)$, where $\Aa\subset \RR^{N\times N}[t]$ is a linear subspace of matrix polynomials such that $\D:\Aa\to  \RR^{N\times N}[t]$.

Using this concept we propose a first method to construct exceptional matrix polynomials. It consists in the following. The starting point is a exceptional matrix weight $\W$
and a linear space $\Xi$ of symmetric differential operators of order at most 2 which are symmetric with respect to the pair $(\W,\Aa)$. We then show (see Lemma \ref{lede}) how to find new weight matrices $\widetilde \W$ for which some operators $D$ in $\Xi$ are still symmetric for the new pair $(\widetilde\W,\Aa)$, from where a family of exceptional polynomials can be constructed: this family of matrix polynomials are orthogonal with respect to $\widetilde\W$ and eigenfunctions of $D$. This illustrates that the phenomenon (2) above also appears in the context of exceptional matrix polynomials.

In Section \ref{qdar} we introduce our main method to construct exceptional matrix polynomials: \textit{quasi-Darboux} transformation. We will use the terminology \textit{quasi-Darboux} transformation if in the scheme of Definition \ref{dxt} above, we factorize $D$ in the form
\begin{equation}\label{qdt}
D(y)=\B(\A(y))+y\Psi(t),
\end{equation}
where $\Psi$ is not a (scalar) constant but a matrix function of $t$.

If one has a second order differential operator $D$ having a sequence of polynomials $(P_n)_n$ as eigenfunctions ($D(P_n)=\Gamma_nP_n$), one can then take any eigenfunction $P$ of $D$ (no matter whether its eigenvalue is scalar or not) as seed function and define
\begin{equation}\label{aop}
\A(y)=-yU+y'(P')^{-1}PU,
\end{equation}
where $U$ is any smooth matrix function. In doing that, there is a first order differential operator $\B$ and a matrix function $\Psi$ (which are explicitly defined from $P$, $U$ and the differential coefficients of $D$: see Lemmas \ref{lde} and \ref{conta})
such that the factorization (\ref{qdt})
holds. Moreover, the matrix functions $\Pp_n=\A(P_n)$ satisfy
$$
\D(\Pp_n)=\Gamma_n\Pp_n,
$$
where $\D$ is the second-order differential operator defined by
\begin{equation}\label{aop1}
\D(y)=\A(\B(y))+yP^{-1}P'\Psi(t)(P')^{-1}PU.
\end{equation}
Consider a classical weight matrix $W$ (which can also be a signed weight matrix) and a sequence of orthogonal polynomials $(P_n)_n$ with respect to $W$. We write $\Upsilon$ for the linear space of differential operators of order at most 2 for which the orthogonal polynomials $(P_n)_n$ are eigenfunctions. Let us note that, if the seed $P$ is an eigenfunction of each operator in $\Upsilon$ (for instance, we can take $P=P_k$ for some $k$, but we could make other choices which are not even a polynomial), our method based on the quasi-Darboux transformation allows us to construct the linear space of differential operators of order at most two defined by $\tilde \Upsilon=\{\D : D\in \Upsilon\}$, where $\D$ is defined from $D\in\Upsilon$ as in (\ref{aop1}), such that the sequence of functions $\Pp_n=\A(P_n)$, are eigenfunctions of each operator $\D\in \tilde \Upsilon$.

However, there is still some work to be done before concluding that $\Pp_n$ are exceptional matrix polynomials. This extra work will consist of two steps, and it is related to the additional assumptions on the seed function $P$ mentioned above.

The first step is to guarantee that $\Pp_n$ are polynomials. That of course happens in the simplest case when the seed $P$ is a polynomial of degree $1$ and $U$ is constant. But when $P$ is a polynomial of degree higher than $1$ or even not a polynomial, as we show in Section \ref{s1} (our collection of instructive examples), one can construct polynomials $\Pp_n$ by carefully choosing of the function $U$.

The second step is to find a weight matrix with respect to which the polynomials $\Pp_n$ are orthogonal (and complete). We show that this can be done by finding a particular symmetric differential operator $D$ in $\Upsilon$
$$
D(y)=y'' F_2+y' F_1+yF_0
$$
for which there exists a (not necessarily polynomial) function $P$ in the kernel of $D$ satisfying that it is invertible in the support of $W$ and
$$
(\tfrac12 F_1+P^{-1}P'F_2)W=W(\tfrac12 F_1+P^{-1}P'F_2)^*.
$$
We then show that the polynomials $\Pp_n$ are formally orthogonal with respect to
\begin{equation}\label{ewi}
\W=(P U)^{-1}P'F_2W((P U)^{-1}P')^*,
\end{equation}
where $U$ is any smooth function. There is still the problem of guaranteeing that $\W$ is a weight matrix, i.e., has finite moments of any order and is positive definite. Although both of them are difficult problems, they are just the same problems that appear in the scalar case when using Darboux transformations, and they have to be solved for each particular family (for instance, as far as we know, the problem of finding
necessary and sufficient conditions to guarantee the existence of a positive measure for the most general examples of exceptional Jacobi polynomials is still open, see \cite{Durx3,Durx4}). Surprisingly enough, we show in Section \ref{s1} that in the matrix case we have some more tools to solve the problem of finding a weight matrix for the exceptional polynomials $\Pp_n$.

The Section \ref{s1} of this paper consists of a collection of many illustrative examples which show how powerful our two methods to construct exceptional matrix polynomials are.

Let us consider our first example (Subsection \ref{ss1}). Assume we start from a weight  $W$, a sequence of orthogonal polynomials $(P_n)_n$ with respect to $W$ and a seed $P$ from where we construct the operator $\A$ (\ref{aop}) and the polynomials $\Pp_n=\A(P_n)$. As mentioned above, one of the problems to prove that $\Pp_n$ are exceptional polynomials is to guarantee the existence of a weight with respect to which they are orthogonal (no matter here whether we are dealing with scalar or matrix functions). This is very much related to the existence of singularities in the support of $W$ of the second order differential operator $\D$ for which the $\Pp_n$'s are eigenfunctions. Lemmas \ref{lde} and \ref{conta} show that likely the zeros of $\det P$ are going to be problematic. The identity (\ref{ewi}) for the exceptional weight shows also this problem. For instance, if $P=P_1$ since $P_1$ is orthogonal with respect to $W$, $\det P_1$ vanishes in the convex hull of the support of $W$ and then $\W$ will not be integrable. This is what happens in the scalar case and this is the reason why, for instance, there are no examples of exceptional Hermite polynomials whose degree gaps consist of the singleton $\{1\}$ (we mean here \textit{non-degenerate} exceptional polynomials, see Remark \ref{node}). Similar problems appear if one takes any of the orthogonal polynomials $P_k$ as the seed function. However, this problem can be eventually avoided in the matrix case just by carefully choosing a signed weight matrix as the starting point.
This is the case for the example we construct in Subsection \ref{ss1}, where we start from the matrix of measures
\begin{equation}\label{wbxi}
W_{a,\xi}(t)=e^{-t^2}\begin{pmatrix} 1&at\\ 0&1\end{pmatrix}\begin{pmatrix}
      \xi &0 \\
      0 & 1
   \end{pmatrix}\begin{pmatrix} 1&0\\ at &1\end{pmatrix},\quad \xi, a\in \RR,
\end{equation}
which it is not positive definite if $\xi<0$. However $W_{a,\xi}(t)dt$ has associated a sequence of orthogonal polynomials ($P_{n,a,\xi})_n$ (see (\ref{bap})), but the norm of $P_{n,a,\xi}$ is not positive definite when $\xi<0$ and
$$
2\xi+na^2<0<2\xi+(n+1)a^2.
$$
In particular, for $\xi(a)=1-a^2$, we have
\begin{equation}\label{fsf}
P_{1,a,\xi(a)}(t)=\begin{pmatrix}
      2t & -a \\
      -a & t(2-a^2)
    \protect\end{pmatrix},
\end{equation}
and then $\det (P_{1,a,\xi(a)})=2(2-a^2)t^2-a^2$ has not real zeros for  $a^2>2$. Using $P_{1,a,\xi(a)}$ as seed function, we can then apply our quasi-Darboux transformation method to construct a sequence $\Pp_n=\A(P_{n,a,\xi(a)})$, $n\in \NN\setminus \{1\}$, $a^2>2$, of exceptional polynomials which are orthogonal (and complete) with respect to the weight matrix $\W_a(t)dt$, $t\in \RR$, where:
\begin{equation}\label{wexi}
\W_a(t)=\frac{e^{-t^2}}{\mathfrak p^2(t)}\begin{pmatrix}
\frac{a^2}{4(a^2-2)}\mathfrak p(t)^2+\frac{\mathfrak p(t)}{a^2-2}-a^2
& a\left(\frac{\mathfrak p(t)}{a^2-2}-2\right)t\\
    a\left(\frac{\mathfrak p(t)}{a^2-2}-2\right)t& \frac{2(a^2(a^2-2)-\mathfrak p(t))}{(a^2-2)^2}
    \protect\end{pmatrix},
\end{equation}
and $\mathfrak p(t)=\det P_{1,a,\xi(a)}(t)\not =0$, $t\in \RR$ (let us remark that $0$ is included in $\NN=\{0,1,2,\dots\}$). It turns out that although $W_{a,\xi(a)}(t)dt$ (\ref{wbxi}), $a^2>2$, is not a weight matrix (because $W_{a,\xi(a)}$ is not positive definite), $\W_a(t)dt$ (\ref{wexi}) is however a weight matrix!

In the scalar setting, it is well-known that the same family of exceptional polynomials can be constructed from the same classical family using different seed functions and a Darboux process which eventually needs a different number of steps.
In a certain sense, this is a consequence of the following property: all the examples of exceptional polynomials have many determinantal representations, which in turn, is closely related to some impressive invariance properties satisfied by Wronskian and Casorati determinants whose entries are orthogonal polynomials belonging to the Askey and $q$-Askey schemes which have been discovered in the last years (see  \cite{Be,BK,du00,du01,duct,duar,ducu,FHV,GUGM2,OS1}).
Hence, it should not be surprising to find such phenomena also in the matrix setting. But as usual, some subtleties appear in the matrix case, probably caused by the existence of an enormous number of examples of classical families of matrix orthogonal polynomials. Indeed, we show how to construct our first example above starting from the following quite different seed function
$$
P(t)= \begin{pmatrix} t & \frac{a}{a^{2}-2} \\ -\frac{a}{2} & t \end{pmatrix} W_{a,1}(t)^{-1} = \begin{pmatrix}-\frac{2te^{t^{2}}}{a^{2}-2} & \frac{a(2t^{2}+1)e^{t^{2}}}{a^{2}-2} \\ -\frac{a(2t^{2}+1)e^{t^{2}}}{2} & \frac{t(2a^{2}t^{2}+a^{2}+2)e^{t^{2}}}{2}
\end{pmatrix}
$$
(compare with (\ref{fsf})). The difference with the scalar case is that, we have to change not only the seed function, but also the initial classical family: instead of using the polynomials $P_{n,a,\xi(a)}$, $\xi(a)=1-a^2$, orthogonal with respect to the signed weight matrix $W_{a,\xi(a)}$, we have to use the orthogonal polynomials $P_{n,a,1}$ orthogonal with respect to the weight matrix $W_{a,1}$.

In Subsection \ref{s12} we construct an example starting from the weight matrix $W_{a,1}$ using a two-step Darboux transformation choosing as seed functions the orthogonal polynomial $P_1$, in the first Darboux step, and  $\Pp_2=\A(P_2)$ in the second Darboux step. We also show an alternative way to produce this example in only a one-step Darboux transformation with a non-polynomial seed function, but starting from a different weight matrix: $W_{a\sqrt 2/\sqrt{3a^2+2},1}$.

In Subsections \ref{lag} and \ref{jac-nonpol}, we show that our method also works for Laguerre and Jacobi matrix weights, respectively. In the first case, we produce an example of exceptional Laguerre matrix polynomials. It turns out that they are also eigenfunctions of a symmetric fifth order differential operator. Although there are no known examples of scalar valued exceptional polynomials having an odd order symmetric differential operator, this phenomenon is not surprising in the matrix setting because precisely the Laguerre matrix weight from where we construct the exceptional Laguerre example has a symmetric third order differential operator (see \cite{DuDI1}).

The last example is devoted to illustrating how our method, developed in Lemma \ref{lede}, can be used to construct examples of exceptional matrix polynomials without using Darboux or quasi-Darboux transforms (see Subsection \ref{sdel}).
Indeed, for the example $\Pp_{n,a}$ in Subsection \ref{ss1} we have found, using our method of quasi-Darboux transform, a five dimension linear space $\Upsilon$ of differential operators of order at most $2$ such that $\Pp_{n,a}$ are eigenfunctions of each operator in $\Upsilon$. Using Lemma \ref{lede}, we have proved that for each $\zeta\ge 0$, there exists exceptional matrix polynomials $\Pp_{n,a,\zeta}$, $n\not =1$, orthogonal with respect to the weight matrix
$$
\frac{1}{\sqrt\pi}\W_a+\zeta \begin{pmatrix}0&0\\0&1\end{pmatrix} \delta_0
$$
and eigenfunctions of the second order differential operator (let us note that $\D$ does not depend on $\zeta$)
\begin{align*}
\D (y)&=y\F_0+y'\F_1+y''\F_2,\\
\F_2(t)&=\frac{1}{\mathfrak p(t)}\begin{pmatrix}
     -a^2(2t^2+1)& -(a^2-2)at(2t^2+1) \\
     \frac{4at}{a^2-2} & 4t^2
    \protect\end{pmatrix},\\
    \F_1(t)&=\frac{1}{\mathfrak p^2(t)}\begin{pmatrix}
     -2t(\mathfrak p^2(t)+4\mathfrak p(t)-4a^2) & \frac{4(a^2-2)t^2}{a}(-\mathfrak p(t)+2a^2) \\
     \frac{4(\mathfrak p(t)+2a^2)(\mathfrak p(t)-a^2)}{a(a^2-2)} & -8a^2t
    \protect\end{pmatrix},\\ \F_0&=\begin{pmatrix}2-\frac{4}{a^2}&0\\0&0\end{pmatrix},
\end{align*}
where $\mathfrak p(t)=\det P_{1,\xi(a)}(t)\not =0$, $t\in \RR$ (assuming $a^2>2$).
Moreover, we guess that $\Pp_{n,a,\zeta}$ cannot be constructed by applying a Darboux or a quasi-Darboux transformation to a sequence of orthogonal matrix polynomials.

Although a discussion on whether the exceptional matrix polynomials satisfy some kind of recurrence relations is beyond the scope of this paper, in the last Section \ref{ssrr}, we include a
couple of comments on this subject. It is known that exceptional scalar valued polynomials $p_n$, $n\in \NN \setminus X$, do satisfy recurrence relations of the form
$$
q(t)p_n(t)=\sum_{j=-r}^r a_{n,j}p_{n-j}(t),
$$
where $q$ is a certain scalar valued polynomial of degree $r$ and $(a_{n,j})_{n}$, $j=-r,\cdots ,r$, are sequences of numbers independent of $t$  (see \cite{STZ,durr,MiTs,GUKKM,durr2} and references therein).
We just point out that this is the case of all the examples constructed in this paper. For instance, we have checked using Maple and Mathematica that the polynomials $\Pp_n$ constructed in Section \ref{ss1}
satisfy a seven-term recurrence relations of the form
$$
q(t)\Pp_{n}(t)=\sum_{j=-3}^3 A_{n,j,\zeta}\Pp_{n+j}(t),
$$
where $q$ is a scalar value polynomial of degree $3$ and satisfying $q'(t)=\det(P_{1,1-a^2})$. The pattern is similar for the other examples considered in this paper.

\medskip

We emphasize something that will be apparent to any reader of this paper:
the full picture of exceptional matrix polynomials is still far from being complete. In this paper, we have developed some tools to construct examples using the quasi-Darboux transformation (in one or several steps), and by constructing a collection of them to illustrate the richness of the situation at hand. The description of how to implement an arbitrary number of quasi-Darboux transforms (likely using quasi-determinants) is very much related to the problem of how to choose an arbitrary number of seeds without introducing singular points in the domain of the differential operators. Both problems remain as challenges that will probably require new ideas.

\section{Preliminaries}\label{pre}

\subsection{Orthogonal matrix polynomials}
We include here some basic definitions on orthogonal matrix polynomials.

A (square) matrix polynomial $P$ of size $N$ and degree $n$ is a
square matrix of size $N\times N$ whose entries are polynomials in
$t\in \RR$ (with real coefficients) of degree less than or
equal to $n$ (with at least one entry of degree $n$):
$$
P(t)=
\begin{pmatrix}
p_{11}(t) & \cdots & p_{1N}(t)\\
\vdots & \ddots & \vdots \\
p_{N1}(t) & \cdots & p_{NN}(t)
\end{pmatrix} ,
$$
or, equivalently, a polynomial of the form
$$
P(t)=A_n t^n+A_{n-1}t^{n-1}+\cdots+A_0,
$$
where $A_0, \dots , A_n$ are matrices of size $N\times N$ and real entries, $A_n\not =0 $. We denote the linear
space of (square) matrix polynomials by $\RR ^{N\times N}[t]$.

A matrix of measures
$$
W=
\begin{pmatrix}
w_{11} & \cdots & w_{1N} \\
\vdots & \ddots & \vdots \\
w_{N1} & \cdots & w_{NN}
\end{pmatrix},
$$
on the real line is a square matrix of size
$N\times N$ whose entries $w_{i,j}$, $i,j=1,\cdots, N$, are Borel complex measures on the real
line.

We say that the matrix of measures $W$ is positive definite if for any Borel set $\Omega $ the matrix $W(\Omega )$ is positive semidefinite, and positive definite at least for one Borelian.

We say that a matrix of measures $W$ is a weight matrix if it is positive definite and has finite moments of any order $n\in \NN$ (as usual,  the  $n$-th moment of $W$ is defined by the integral $\int _\RR t^ndW(t)$).
If the Hermitian matrix of measures $W$ has finite moments of any order $n\in \NN$ but it is not positive definite we say (abusing language) that $W$ is a signed weight matrix.

The support of a matrix of measures $W$ is by definition
$$
\supp (W)=\{ x\in \RR : W((x-\epsilon,x+\epsilon))\not =0 \text{ for any $\epsilon>0 \}$}.
$$

A Hermitian sesquilinear form in the linear space $\RR ^{N\times N}[t]$ of matrix polynomials can be associated
to a  weight matrix (signed of not):
\begin{align}\label{bf}
\langle P,Q\rangle =\int _\RR P(t)dW(t)Q^*(t)\in \RR ^{N\times N}.
\end{align}
If $W$ is a weight matrix, then $\langle P,P\rangle$ is positive definite for any polynomial $P\not =0$.

Let us notice that the definition of the Hermitian sesquilinear form above gives \textsl{priority} to the multiplication on the left with respect to that on the right. Indeed, we have $\langle AP,Q\rangle =A\langle P,Q\rangle $,
$\langle P,AQ\rangle =\langle P,Q\rangle A^*$
but, in general, $\langle PA,Q\rangle \not =A\langle P,Q\rangle $,
$\langle P,QA\rangle \not =\langle P,Q\rangle A^*$.

We say that a sequence of matrix polynomials $(P_n)_n$ is
orthogonal with respect to the weight matrix (signed or not) $W$ if each polynomial $P_n$ has degree
$n$ with nonsingular leading coefficient and they are orthogonal
with respect to $\langle \cdot, \cdot \rangle $, i.e., $\langle
P_n,P_k\rangle =\Gamma _n\delta _{n,k}$, with $\Gamma_n$
nonsingular. The Gram-Schmidt orthogonalizing method guarantees the existence of orthogonal polynomials with respect to a weight matrix (in general, this is not the case for a signed weight matrix).
If $\Gamma _n=I$ ($I$ stands for the $N\times N$ identity matrix), we say that the sequence
$(P_n)_n$ is orthonormal with respect to $W$. The theory of matrix valued orthogonal
polynomials was started by M. G. Krein in 1949 \cite{Kre1,Kre2} (see also \cite{Ber}, \cite{Atk} or \cite{DPS}).

As in the scalar case, if the sequence $(P_n)_n$ is orthonormal with respect to $W$, then it satisfies
the so called three term recurrence formula
\begin{align}\label{defttrr}
tP_n(t)=A_{n+1}P_{n+1}(t)+B_nP_n(t)+A_n^*P_{n-1}(t),
\end{align}
where we take $P_{-1}=0$. In this formula the coefficients
$B_n$ have to be Hermitian, and  the coefficients $A_{n}$ have to be nonsingular.

\bigskip

We say that $W$
reduces (to smaller size weights) if there
exists a nonsingular  matrix $T$ independent of $t$ for which
$$
W(t)=TD(t)T^*,
$$
with $D$ a block diagonal (weight) matrix. If $D$ is a diagonal matrix, one says that $W$ reduces to scalar weights.

A convenient way of checking if a weight matrix $W=W(t)dt$ reduces to scalar weights is the following (\cite{DuGr1}):

\begin{Lem}\label{lemnre}
Assume that
the matrix of functions $W(t)$ satisfies $W(a)=I$,
for some real number $a$. Then $W$ reduces to scalar weights if and only if
$W(t)W(s)=W(s)W(t)$ for all $t,s$ in the support of $W$.
\end{Lem}

The condition $W(a)=I$, for some real number $a$, is not an important restriction.
Indeed, as far as one has that $W(a)$ is nonsingular for
some $a$, we can take $\tilde W=\left(W(a)\right)^{-1/2}W\left(W(a)
\right)^{-1/2}$, so that $\tilde W(a)=I$,
and it is clear that $W$ reduces to scalar weights if and only if $\tilde W$ do so.
For general criteria of reducibility for weight matrices, the reader may also consult \cite{TZ18}.

None of the examples of weight matrices considered in Section \ref{s1} of this paper reduces.

\color{black}

\bigskip
\begin{Rem}\label{span} Abusing language, given polynomials $P_n$, $n\in \NN$, the linear space defined by
$$
\left\{\sum_{n\in X, X{ \mbox{\footnotesize finite}}}A_nP_{n}: X\subset \NN, A_n\in \RR^{N\times N}\right\},
$$
will be called the linear span of the polynomials $P_n$, $n\in \NN$.
\end{Rem}

\bigskip

\begin{Rem}\label{rdet}

Throughout this paper, we will use some properties of determinate weight matrices. A weight matrix $W$ is determinate if there is no other weight matrix with the same moments as those of $W$ (see, for instance, \cite{DuLo1}). Using moment problem standard techniques, it is easy to prove that if the Fourier transform $H(z)$ of $W$, defined by $H(z)=\int e^{-itz}dW(t)$, is analytic around $z=0$ (i.e. the entries of $H(z)$ are analytic around $z=0$) then the weight matrix $W$ is determinate. We also point out that for a determinate weight matrix $W$, the linear space of polynomials is always dense in $L^2(W)$. For the definition of $L^2(W)$ see \cite{DuLo1}.

\end{Rem}

\subsection{Symmetric operators}

The more important examples of orthogonal matrix polynomials $\{P_n\}$ are those that, for a given second-order differential operator $D$ acting on the right-hand side
\begin{equation}\label{difope}
D(y)=y''(t)F_2(t)+y'(t)F_1(t)+y(t)F_0(t),
\end{equation}
satisfy $ D(P_n)=\Gamma_nP_n$ for all $n$, where the differential coefficients $F_2$, $F_1$ and $F_0$ are
matrix polynomials (which do not depend on $n$) of degrees not
higher than $2$, $1$ and $0$, and $\Gamma_n$ are matrices.

These families are most likely going to play in the matrix setting a role similar to that of the classical families of Hermite, Laguerre, or Jacobi in the scalar case.

The key concept for constructing this kind of examples is that of the symmetry of a differential operator with respect to a weight matrix:

\begin{Def}\label{soder}
We say that a differential operator $D$
 is symmetric with respect to the signed weight matrix $W$ if
$$
\int D(P)WQ^*dt=\int PW(D(Q))^*dt,
$$
for any matrix polynomials $P,Q$.
\end{Def}

The basic idea for the study of the symmetry of a second order
differential operator
with respect to a weight matrix $W$ is
to convert this condition of symmetry into a set of differential
equations relating the weight matrix and
the differential coefficients of the differential operator (see \cite{DuGr1,GPT3}):

\begin{Theo} \label{thcondiciones2}
Let $W=W(t)dt$ be a signed weight matrix with support in an interval of the real line (bounded or unbounded).
Assume that $W$ is twice differentiable in the interior $(a,b)$ of its support, and the boundary conditions that
\begin{equation}\label{boundaryconditions}
\lim _{t\to a^+,b^-}t^nF_2(t)W(t)=0 \mbox{ and } \lim _{t\to a^+,b^-}t^n\left[(F_2(t)W(t))'- F_1(t)W(t))\right]=0,
\end{equation}
for any $n\ge 0$. If the weight matrix $W$ satisfies
\begin{equation}\label{locdeq3.1}
F_2W=WF_2^*,
\end{equation}
as well as
\begin{equation}\label{locdeq3.2}
2(F_2W)'=WF_1^*+F_1W,
\end{equation}
and
\begin{equation}\label{locdeq3.3}
(F_2W)''-(F_1W)'+F_0W=WF_0^*,
\end{equation}
then the second order differential operator $D$ (see (\ref{difope})) is symmetric with respect to $W$.
\end{Theo}

\section{Exceptional matrix polynomials}\label{exc}
As we wrote in the Introduction, in the scalar case, exceptional polynomials $p_n$, $n\in \NN \setminus X$ and $X$ with finitely many elements, are complete orthogonal polynomial systems with respect to a positive measure in the real line which in addition are eigenfunctions of a second order differential operator (whose differential coefficients can then be rational functions).
This suggests the following definition for exceptional matrix polynomials.

\begin{Def}\label{dxmp} Let $X$ be a finite subset of nonnegative integers. We say that the matrix polynomials $P_n$, $n\in \NN\setminus X$, $P_n$ of degree $n$ with nonsingular leading coefficient, are exceptional (orthogonal) matrix polynomials if
\begin{enumerate}
\item They are orthogonal with respect to a weight matrix $\W$, in the sense that
$$
\int _\RR P_n(t)d\W(t)P_k^*(t)=\Gamma _n\delta _{n,k},\quad n,k\in \NN\setminus X,
$$
with $\Gamma_n$ positive definite.
\item They are complete in $L^2(\W)$.
\item They are also left eigenfunctions of a right-hand side
second-order differential operator of the form
\begin{equation}\label{xdo}
\D(y)=y''(t)\F_2(t)+y'(t)\F_1(t)+y(t)\F_0(t),
\end{equation}
where the differential coefficients $\F_2$, $\F_1$ and $\F_0$ are rational functions of $t$.
\end{enumerate}
\end{Def}

Given a differential operator $\D$, one of the key concepts to construct exceptional matrix polynomials is that of the symmetry of $\D$ with respect to the pair $(\W, \Aa)$, where $\W$ is a weight matrix and $\Aa$ is a (left) linear subspace of the linear space of matrix polynomials $\RR ^{N\times N}[t]$.

\begin{Def}\label{swa} Let $\W$ be a signed weight matrix and let $\Aa $ be a linear subspace of  the linear space of matrix polynomials $\RR ^{N\times N}[t]$.
We say that a differential operator $\D:\Aa\to \RR ^{N\times N}[t]$ is symmetric with respect to the pair $(\W , \Aa)$ if
$\langle \D(P),Q\rangle _\W =\langle P,\D(Q)\rangle _\W$ for all polynomials $P, Q \in \Aa$, where the bilinear form $\langle \cdot, \cdot \rangle _\W $ is defined by (\ref{bf}).
\end{Def}

When $\Aa=\RR ^{N\times N}[t]$, we get the usual symmetry defined in Definition \ref{soder}.

The following lemma will be useful.

\begin{Lem}\label{lsyo} Let $\D$ be a symmetric operator with respect to the pair $(\W ,\Aa)$. Let $X$ be a finite set of nonnegative integers such that for each $n\in \NN\setminus X$
we have a polynomial $R_n\in \Aa$ of degree $n$ with nonsingular leading coefficient that is an eigenfunction of the operator $\D$ with eigenvalues $\Gamma_n\in \RR ^{N\times N}$: $\D(R_n)=\Gamma_n R_n$.
\begin{enumerate}
\item If we assume in addition that the square matrices $\Gamma_n$ and $\Gamma_m$ do not have common eigenvalues for $n\not =m$. Then the polynomials $R_n$, $n\in \NN\setminus X$, are orthogonal with respect to $\W$.
\item In any case, there exist polynomials $\Pp_n\in \Aa$, $n\in \NN\setminus X$, which are orthogonal with respect to $\W$ and eigenfunctions of $\D$.
\end{enumerate}
\end{Lem}

\begin{proof}
Part (1) is an easy consequence of the necessary and sufficient conditions for the existence of solutions of the matrix equation $AX-XB$ ($X$ is the unknown and $A$ and $B$ the data; see \cite[Chapter VIII]{Gan}).

The proof of Part (2) is as follows. By applying the Gram-Schmidt orthogonalizing method to the polynomials $R_n$, $n\in \NN\setminus X$, we construct polynomials $\Pp_n\in \Aa$ of degree $n$ with non singular leading coefficient and orthogonal with respect to $\W$. Obviously, the polynomials $\Pp_n$ and $R_n$, $n\in \NN\setminus X$, have the same linear span (see Remark \ref{span}), which we denote by $\Bb$. Hence, if we set $Y_n=\{j\in \NN\setminus X: j\le n\}$, we have $\Pp_n(t)=\sum_{j\in Y_n}A_jR_j$ and hence
$$
\D(\Pp_n)=\sum_{j\in Y_n}A_j\Gamma_jR_j\in \Bb,
$$
from where we deduce that $\D(\Pp_n)=\sum_{j\in Y_n}B_j\Pp_j$.
Hence, using the symmetry of $\D$ and the orthogonality of $\Pp_j$, we deduce for $0\le k<n$, $k\in Y_n$:
$$
B_k\langle \Pp_k,\Pp_k\rangle=\langle \D(\Pp_n),\Pp_k\rangle=\langle \Pp_n,\D(\Pp_k)\rangle =0.
$$
So, $\Pp_n$ are eigenfunctions of $\D$.
\end{proof}

Starting from a symmetric operator $\D$ with respect to the pair $(\W ,\Aa)$, Part (2) of the previous Lemma provides a first method to construct new families of exceptional polynomials $\Pp_n$, $n\in \NN\setminus X$.

\begin{Lem}\label{lede} Let $\W$ be a weight matrix going along with a sequence of exceptional polynomials $R_n$, $n\in \NN\setminus X$, orthogonal with respect to $\W$ and
a second order differential operator $\D$ (of the form (\ref{xdo})) which is symmetric with respect to the pair $(\W ,\Aa)$, where $\Aa$ is the linear span of $R_n$, $n\in \NN\setminus X$ (see Remark \ref{span}).
Assume that there exists $t_0\in \RR$ and a positive semidefinite matrix $M$ such that
$$
\F_2(t_0)M=0,\quad \F_1(t_0)M=0,\quad \F_0(t_0)M=M(\F_0(t_0))^*.
$$
Then for each $\xi\ge 0$, there exist exceptional polynomials $\Pp_{n,\xi}$, $n\in \NN\setminus X$, orthogonal with respect to the weight matrix
$$
\W+\xi M\delta_{t_0},
$$
which are also eigenfunctions of $\D$.
\end{Lem}

\begin{proof}
Since $\D$ is symmetric with respect to the pair $(\W ,\Aa)$, we have for $P,Q\in \Aa$
\begin{align*}
\langle \D(P),Q&\rangle_{\W+\xi M\delta_{t_0}}\\
&=\langle \D(P),Q\rangle_{\W}+\xi [P''(t_0)\F_2(t_0)+P'(t_0)\F_1(t_0)+P(t_0)\F_0(t_0)]MQ^*(t_0)\\&=
\langle P, \D(Q)\rangle_{\W}+\xi P(t_0)M[Q''(t_0)\F_2(t_0)+Q'(t_0)\F_1(t_0)+Q(t_0)\F_0(t_0)]*\\&=\langle P,\D(Q)\rangle_{\W+\xi M\delta_{t_0}}.
\end{align*}
That is, for $\xi\ge 0$,  $\D$ is also symmetric with respect to the pair the pair $(\W +\xi M\delta_{t_0},\Aa)$. It is enough to apply Part (2) of  Lemma \ref{lsyo}.
\end{proof}

\bigskip

The following Theorem provides sufficient conditions to guarantee the symmetry of a second order differential operator $\D$ of the form (\ref{xdo}) with respect to the pair $(\W,\Aa)$. The proof is similar to that of Theorem \ref{thcondiciones2} and it is omitted.

\begin{Theo} \label{sywa}
Let $\W=\W(t)dt$ be a signed weight matrix with support in an interval $[a,b]$ of the real line (bounded or unbounded). Let $\D$ be a second order differential operator as in (\ref{xdo}) and $\Aa$ a linear subspace of $\RR ^{N\times N}[t]$ for which $\D(P)\in \RR ^{N\times N}[t]$ if $P\in \Aa$.
Assume that $\W$ is twice differentiable in $(a,b)$, and satisfies the boundary conditions
(\ref{boundaryconditions}) and the identities (\ref{locdeq3.1}), (\ref{locdeq3.2}) and (\ref{locdeq3.3}).
Then the second order differential operator $\D$  is symmetric with respect to the pair $(\W,\Aa)$.
\end{Theo}

\bigskip

\begin{Rem}\label{node}
Given polynomials $P_n$, $n\in \NN\setminus X$, which are eigenfunctions of a second order differential operator of the form (\ref{xdo}), it is obvious that for a polynomial $Q$, the polynomials
$P_nQ$ are also eigenfunctions of a second order differential operator of the form (\ref{xdo}). And conversely, if polynomials $P_n$, $n\in \NN\setminus X$, are eigenfunctions of a second order differential operator of the form (\ref{xdo}) and they have a common factor $Q$, then the polynomials $P_n Q^{-1}$
are also eigenfunctions of a second order differential operator of the form (\ref{xdo}). Hence, for obvious reasons, when constructing exceptional matrix polynomials, the \textit{degenerate} situation of exceptional polynomials $P_n$ having a polynomial $Q$ as a common factor should be avoided.
\end{Rem}

\section{Darboux and quasi-Darboux transformation}\label{qdar}

The main tool we will use to construct exceptional matrix polynomials is that of quasi-Darboux transformation.
As we wrote in the Introduction, we will use the terminology \textit{quasi-Darboux} transformation if in the scheme of a Darboux transformation displayed in Definition \ref{dxt}, we factorize the second-order differential operator $D$ in the form
$$
D(y)=\B(\A(y))+y\Psi(t),
$$
where $\Psi$ is  a  matrix function of $t$. From this factorization, we define a new second order differential operator $\D$ as
$$
\D(y)=\A(\B(y))-yA_1^{-1}\Psi A_1,
$$
where $A_1$ is the coefficient of $d/dt$ of $\A$. In this case, if $P$ is an eigenfunction of $D$, the function $\A(P)$ need not be an eigenfunction of $\D$.

The following two Lemmas will show the usefulness of quasi-Darboux transformations to construct exceptional matrix polynomials. In the first one,
we find necessary and sufficient conditions on the function $\Psi$, so that the function $\A(P)$ is also an eigenfunction of $\D$ if $P$ is an eigenfunction of $D$. In the second Lemma, we show how to choose the differential coefficients of $\A$ so that these necessary and sufficient conditions on $\Psi$ always hold.

\begin{Lem}\label{lde}
Let $\A$ and $D$ be first and second order differential operators of the form
\begin{align}\label{eqp1}
\A(y)&=yA_0+y'A_1,\\\label{eq1}
D(y)&=yF_0+y'F_1+y''F_2,
\end{align}
where $A_i$ and $F_i$ are matrix functions defined in an open set $\Omega$ of the real line satisfying the necessary regularity conditions so that the following definitions make sense.
We define the first order differential operator $\B$  as
\begin{equation}\label{eq2}
\B(y)=yB_0+y'B_1,\quad B_1=A_1^{-1}F_2,\quad B_0=A_1^{-1}(F_1-A_0B_1-A_1'B_1).
\end{equation}
We can then factorize the second order differential operator $D$ in the form
$D(y)=\B(\A(y))-y\Psi $, where
\begin{equation}\label{eq3c}
\Psi=\B(A_0)-F_0.
\end{equation}
We produce then the second order differential operator
\begin{equation}\label{eq3z}
\D(y)=\A(\B(y))-yA_1^{-1}\Psi A_1
\end{equation}
with explicit expression
$$
\D(y)=y(\A(B_0)-A_1^{-1}\Psi A_1)+y'(\A(B_1)+B_0A_1)+y''B_1A_1.
$$
Let $P$ be an eigenfunction of $D$, $D(P)=\Gamma P$,
and define the function $\Pp=\A(P)$.
Assume that $\det P(t)\not =0$, $t\in \Omega$. Then $\Pp$ is an eigenfunction of $\D$ with the same eigenvalue, $\D(\Pp)=\Gamma \Pp$, if and only if
\begin{equation}\label{eq3f}
\A(\Psi )=A_0A_1^{-1}\Psi A_1.
\end{equation}
\end{Lem}

\begin{proof}
A simple computation gives
$$
\B(\A(y))=y''A_1B_1+y'(A_1B_0+A_0B_1+A_1'B_1)+y(A_0B_0+A_0'B_1).
$$
The first part follows easily by comparing to (\ref{eq1}).

Again, a simple computation using (\ref{eq2}) gives
\begin{align*}
\D(y)&=y(B_0A_0+B_0'A_1-A_1^{-1}\Psi A_1)+y'(B_1A_0+B_1'A_1+B_0A_1)+y''B_1A_1\\
&=y(\A(B_0)-A_1^{-1}\Psi A_1)+y'(\A(B_1)+B_0A_1)+y''B_1A_1.
\end{align*}

Finally, setting $\tilde \Psi =A_1^{-1}\Psi A_1$, and using the assumption (\ref{eq3f}), we have
\begin{align*}
\D(\Pp)&=\A(\B(\A(P))) -\A (P)\tilde \Psi =\A(D(P)+P\Psi ) -\A (P)\tilde \Psi \\
&=\A(\Gamma P+P\Psi ) -\A (P)\tilde \Psi \\
&=\Gamma\A(P)+(P\Psi )A_0+(P\Psi )'A_1-\A (P)\tilde \Psi \\
&=\Gamma\Pp+P\Psi A_0+(P_n'\Psi +P\Psi '-P'\Psi )A_1-PA_0\tilde \Psi \\
&=\Gamma\Pp+P[\Psi A_0+\Psi 'A_1-A_0\tilde \Psi ].
\end{align*}
Since $\det P(t)\not =0$, $t\in \Omega$, we deduce that $\D(\Pp)=\Gamma \Pp$ if and only if the identity (\ref{eq3f}) holds.
\end{proof}

\begin{Lem}\label{conta}
Let $P$ be an eigenfunction of the second order differential operator (\ref{eq1}), $D(P) = \Gamma P$, and $U$ be any matrix function ($F_i$, the differential coefficients of $D$, $P$ and $U$ defined in an open set $\Omega$ of the real line and satisfying the necessary regularity conditions so that the following definitions make sense).
Define $A_0=-U$, $A_1=(P')^{-1}PU$, and the first order differential operator
$$
\A(y)=-yU+y'(P')^{-1}PU
$$
(so that $\A(P)=0$). Then the function $\Psi$ defined as in (\ref{eq3c}) of Lemma \ref{lde} satisfies the condition (\ref{eq3f}). Hence the function $\A(P)$ is an eigenfunction of $\D$ (\ref{eq3z}) with the same eigenvalue as $P$: $\D(\Pp)=\Gamma \Pp$.
\end{Lem}

\begin{proof}

From Lemma \ref{lde}, we have $A_0=-U$, $A_1=(P')^{-1}PU$. Then
\begin{align*}
B_0&=U^{-1}P^{-1}P'[F_1+UU^{-1}P^{-1}P'F_2-((P')^{-1}PU)'U^{-1}P^{-1}P'F_2]\\
&=U^{-1}P^{-1}[P'F_1+P'P^{-1}P'F_2\\ &\quad\quad -P'(((P')^{-1})'PU+(P')^{-1}P'U+(P')^{-1}PU')U^{-1}P^{-1}P'F_2]\\
&=U^{-1}P^{-1}[P'F_1-P'((P')^{-1})'P'F_2-PU'U^{-1}P^{-1}P'F_2].
\end{align*}
Since $((P')^{-1})'P'=-(P')^{-1}P''$ and the eigenvalue equation $D(P)=\Gamma P$ gives $P''F_2+P'F_1=-PF_0+\Gamma P$, we finally get
\begin{align}\nonumber
B_0&=U^{-1}P^{-1}[P'F_1-P'((P')^{-1})'P'F_2-PU'U^{-1}P^{-1}P'F_2]\\\label{net}
&=U^{-1}P^{-1}(\Gamma P-PF_0-PU'U^{-1}P^{-1}P'F_2).
\end{align}
Then,  we have
\begin{align}\label{der}\nonumber
\Psi &=-UB_0-U'B_1-F_0\\\nonumber
&=-UU^{-1}P^{-1}(\Gamma P-PF_0-PU'U^{-1}P^{-1}P'F_2)-U'U^{-1}P^{-1}P'F_2-F_0\\&=-P^{-1}\Gamma P.
\end{align}
On the one hand, using that $(P^{-1})'=-P^{-1}P'P^{-1}$, we deduce
\begin{align*}
\A(\Psi)&=P^{-1}\Gamma PU-(P^{-1}\Gamma P)'(P')^{-1}PU\\
&=P^{-1}\Gamma PU-((P^{-1})'\Gamma P+P^{-1}\Gamma P')(P')^{-1}PU\\
&=P^{-1}P'P^{-1}\Gamma P(P')^{-1}PU.
\end{align*}
On the other hand, we have
$$
A_0A_1^{-1}\Psi A_1=UU^{-1}P^{-1}P'P^{-1}\Gamma P(P')^{-1}PU=P^{-1}P'P^{-1}\Gamma P(P')^{-1}PU.
$$
This shows that
$$
\A(\Psi)=A_0A_1^{-1}\Psi A_1.
$$
\end{proof}

\medskip
Let us now take a weight matrix $W$ (which can also be a signed weight matrix) having a sequence of orthogonal matrix polynomials $(P_n)_n$ with respect to $W$. As in the Introduction, write $\Upsilon$ for the linear space of differential operators of the form (\ref{difope}) for which the orthogonal polynomials $(P_n)_n$ are eigenfunctions.
Consider then a common eigenfunction $P$ of all the operators in $\Upsilon$ (for instance, we can take $P=P_k$ for some $k$, but we could make other choices which are not even a polynomial). For a function $U$, consider the first order operator $\A$ defined from $P$ as in Lemma \ref{conta}:
$$
\A(y)=-yU+y'(P')^{-1}PU.
$$
This operator only depends on $P$ and $U$. Hence, using Lemmas \ref{lde} and \ref{conta}, we deduce that the functions $\Pp_n=\A(P_n)$ are eigenfunctions of each second-order operator $\D$ constructed by applying a quasi-Darboux transformation to the operator $D\in \Upsilon$.

\medskip

As we wrote in the Introduction, there is still some work to be done before concluding that the functions $\Pp_n$ are exceptional matrix polynomials. This extra work will consist in two steps.
The first step is to guarantee that $\Pp_n$ are indeed polynomials. As we show in our collection of examples (Section \ref{ss1}), by a careful choice of the function $U$, we can get $\Pp_n= \A(P_n)$ to be a matrix polynomial with non-singular leading coefficient except for certain exceptional values of $n$.

The second step is to find a weight matrix with respect to which the polynomials $\Pp_n$ are orthogonal (and complete). In the rest of this Section, we will show that by forcing
$\Psi=0$ in Lemma \ref{conta} (that is, $P$ is in the kernel of $D$), it will allow us, under certain additional hypotheses, to explicitly find a signed weight matrix $\W$ (which eventually will be a weight matrix) such that the polynomials $\Pp_n=\A(P_n)$ are orthogonal with respect to $\W$. In order to do that, we need to consider the formal adjoint $E^\dagger$ of a differential operator $E$ with respect to a matrix function $W$. We start by extending the Hermitian sesquilinear form (\ref{bf}) to the linear space of rational functions. Hence, let us consider a symmetric invertible matrix-valued smooth function $W(t)$ for $t\in(a,b)$, with finite moments (and $a<b$ with $a$ and $b$ finite real numbers or infinite). At this stage, we do not need to require $W$ to be positive definite, i.e., $W(t)dt$ is assumed to be a signed weight matrix. Denotes by $\Rr_{a,b}$ the linear space of rational functions with no poles in the closure of $(a,b)$.
$W$ defines a Hermitian sesquilinear form by
$$
\langle p,q\rangle_W =\int _a^b p(t)W(t)q^*(t)dt,
$$
for any $p,q\in \Rr_{a,b}$. For an operator $D$ of the form (\ref{gen-D}) with $F_j\in \Rr_{a,b}$, we say that $D$ is symmetric with respect to $(W,\Rr_{a,b})$ if
$$
\langle D(p),q\rangle_W=\langle p, D(q)\rangle_W,\quad p,q\in \Rr_{a,b}.
$$
Given a differential operator $E$ of order $m$, $E(y)=\sum_{j=0}^m y^{(j)}\,E_j,$ we can also denote it by
\begin{equation}\label{fao}
E=\sum_{j=0}^m\partial^j\,E_j.
\end{equation}
By abuse of notation, the operators of the form $\partial^0 E_0$ will be denoted just by $E_0$, so that $E_0(y)=yE_0$.

\begin{Def}\label{edag} Let $W$ be an invertible matrix function in a certain open set $I\subset \RR$.
For any differential operator $E=\sum_{j=0}^m \partial^j E_j$ we will define its {\it formal adjoint with respect to $W$} as the operator
$$
E^\dagger(y)=\left[W^{-1}\comp \left(\sum_{j=0}^m (-1)^j\partial^j \comp E_j^* \right)\comp W\right](y)=\left(\sum_{j=0}^m(-1)^j \left(y W E_j^*\right)^{(j)} \right)W^{-1}
$$
(where $A\comp B$ denotes the composition of the operators $A$ and $B$).
\end{Def}

This notion of formal adjoint is motivated by the fact that when $W$ is smooth and ``well behaved," it is a straightforward computation to observe that if there exists a differential operator that is the adjoint of $E$ with respect to $\langle\cdot,\cdot\rangle_W$, then it has to be $E^\dagger$.

\begin{Rem}\label{acno}
In this paper, given a differential operator as in (\ref{fao}), we use the \textit{functional} notation $E(y)$ to denote that the operator $E$ sends the function $y$ to the function
$E(y)=\sum_{j=0}^m y^{(j)}\,E_j$.

This can also be done with the following notation (which has the more algebraic flavour of the operator $E$ acting on the right-hand side of the function $y$):
$$
y\cdot E=E(y)=\sum_{j=0}^m y^{(j)}\,E_j.
$$
Using this notation, the operator $E^\dagger$ can be written in the form
$$
E^\dagger=W\cdot\left(\sum_{j=0}^m (-1)^j E_j^*\cdot \partial^j \right)\cdot W^{-1}.
$$
Let us also note that the composition of two differential operators $E$ and $T$, $E\comp T(y)=E(T(y))$, gives in the action notation:
$y\cdot T\cdot E=E\comp T(y)$.
\end{Rem}

The case when $P$  in Lemma \ref{conta} is in the kernel of $D$ allows us to write the operator $D$ in the following \textit{formal} self-adjoint manner.

\begin{Lem}\label{self-adj-nopol} Let $W$ be an invertible matrix function in a certain open set $I\subset \RR$. Let $D$ be the second-order differential operator
\begin{equation}\label{gen-D}
D=\partial^2F_2+\partial F_1+F_0,
\end{equation}
and assume that it satisfies the symmetry equations (\ref{locdeq3.1}) and (\ref{locdeq3.2}), and that there exists a function $P$ in the kernel of $D$ invertible in $I$. Then $D$ can be factorized in the form
\begin{equation}\label{sadj}
D=-(\partial -{P}^{-1}{P}')^\dagger \comp F_2 \comp(\partial -{P}^{-1}{P}')
\end{equation}
if and only if $P$ satisfies
\begin{equation}\label{lus}
    (\tfrac12 F_1+P^{-1}P'F_2)W=W(\tfrac12 F_1+P^{-1}P'F_2)^*.
\end{equation}
Moreover, in that case for any smooth matriz function $U$, invertible in $I$, consider the first order operator $\B$ as in Lemma \ref{conta}: i.e., $\B=\partial B_1+B_0$ given by (see (\ref{net}))
\begin{equation}\label{kt0}
B_1=U^{-1}P^{-1}P'F2,\quad B_0=-U^{-1}F_0-U^{-1}U'U^{-1}P^{-1}P'F2.
\end{equation}
Then
\begin{equation}\label{kt1}
\B=-(\partial -{P}^{-1}{P}')^\dagger \comp (U^{-1}P^{-1}P'F_2).
\end{equation}
\end{Lem}

(Let us note that using the action notation, (\ref{sadj}) can be written in the form
$$
D=-(\partial -{P}^{-1}{P}')\cdot  F_2 \cdot(\partial -{P}^{-1}{P}')^\dagger ).
$$
\begin{proof}
Indeed, write $\tilde D$ for the second order differential operator in the right-hand side of (\ref{sadj}). A simple computation gives
$$
\tilde D(f)-D(f)=f'g-P^{-1}P'g,
$$
where
\begin{equation}\label{fug}
g(t)=-(\tfrac12 F_1+P^{-1}P'F_2)+W(\tfrac12 F_1+P^{-1}P'F_2)^*W^{-1}.
\end{equation}
Hence, if $D=\tilde D$ then the equation (\ref{sadj}) holds.

Assume next that the equation (\ref{sadj}) holds. If we consider the operator $-(\partial -{P}^{-1}{P}')^\dagger \comp (U^{-1} P^{-1} P'F_2)$  applied on a function $f$ we obtain
\begin{equation}\label{dtd}
f'U^{-1}P^{-1}P'F2+f[(U^{1}P^{-1}P'F_2W)'W^{-1}+U^{1}P^{-1}P'F_2W(P^{-1}P')^*W^{-1}].
\end{equation}
Hence (\ref{kt1}) is equivalent to (see (\ref{kt0}))
\begin{align}\label{kt2}
(U^{-1}P^{-1}&P'F_2W)'W^{-1}+U^{-1}P^{-1}P'F_2W(P^{-1}P')^*W^{-1}\\\nonumber&\hspace{2.5cm}=-U^{-1}F_0-U^{-1}U'U^{-1}P^{-1}P'F2.
\end{align}
Taking into account that $(U^{-1})'=-U^{-1}U'U^{-1}$, we then have
$$
(U^{-1}P^{-1}P'F_2W)'W^{-1}=-U^{-1}U'U^{-1}P^{-1}P'F2+U^{-1}(P^{-1}P'F_2W)'W^{-1}.
$$
Using (\ref{kt2}), we get that (\ref{kt1}) is equivalent to
\begin{equation}\label{kt3}
(P^{-1}P'F_2W)'W^{-1}+P^{-1}P'F_2W(P^{-1}P')^*W^{-1}+F_0=0.
\end{equation}
The rest of the proof is just a computation using that $P$ is in the kernel of $D$, $(P^{-1})'=-P^{-1}P'P^{-1}$, and the symmetry equations (\ref{locdeq3.1}) and (\ref{locdeq3.2})
\begin{align*}
&(P^{-1}P'F_2W)'W^{-1}+P^{-1}P'F_2W(P^{-1}P')^*W^{-1}+F_0\\&\quad =
P^{-1}(P''F_2+PF_0)+(P^{-1})'P'F_2+P^{-1}P'[(F_2W)'+F_2W(P^{-1}P')^*]W^{-1}\\&\quad =
P^{-1}P'\left[-F_1-P^{-1}P'F_2 +\frac12F_1+W(\frac12F_1+P^{-1}P'F_2)^*W^{-1}\right]\\&\quad =
P^{-1}P'\left[-\frac12F_1-P^{-1}P'F_2+W(\frac12F_1+P^{-1}P'F_2)^*W^{-1}\right].
\end{align*}
In particular, for $U=P^{-1}P'$, we deduce that the identity (\ref{sadj}) holds.
\end{proof}

We next show the advantage of having an operator $D$ written in the form (\ref{sadj}).

\begin{Lem}\label{pes1}
Let $D$ be the second order differential operator (\ref{gen-D}), satisfying that $F_2W=WF_2^*$, $\lim_{x\to{a,b}}p(F_2 W)^*=0$ for $p\in \Rr_{a,b}$, with $F_2$ invertible on $(a,b)$ and $F_j\in \Rr_{a,b}$.
Assume in addition that $D$ satisfies the factorization \eqref{sadj} where the function ${P}$ is invertible on the closure of $(a,b)$ and ${P}^{-1}{P}'\in \Rr_{a,b}$. Assume also that we have polynomials $P_n$, $n\in \NN\setminus X$, and $X$ finite, with nonsingular leading coefficient, orthogonal with respect to $W$.
Then the operator
$$
\tilde D=-(\partial -{P}^{-1}{P}')\comp(\partial -{P}^{-1}{P}')^\dagger\comp F_2
$$
is symmetric with respect to $(F_2W,\Rr_{a,b})$.
Moreover,  $(P_n'-P_n{P}^{-1}{P}')$, $n\in\mathbb N\setminus X$,  are orthogonal with respect to $F_2W$.
\end{Lem}
\begin{proof}
Let $p, q\in \Rr_{a,b}$, and let us write $\A =\partial -{P}^{-1}{P}'$ and $\tilde W=F_2W$. Then
\begin{align*}
\langle \A(p) ,q\rangle_{\tilde W}
=&\int_a^b (\partial -{P}^{-1}{P}')(p)\,  F_2W   q^*,
\intertext{since $(F_2W)^*=F_2W$, then}
\langle \A(p) ,q\rangle_{\tilde W}
=&\int_a^b  (\partial -P^{-1}P')(p)\,  (qF_2 W)^*.
\intertext{Taking into account that $\lim_{x\to{a,b}}p (F_2 W)^*=0$, by integration by parts we have
}
\langle \A (p),q\rangle_{\tilde W}
=&\int_a^b p\,  [(-\partial -(P^{-1}P')^*)(qF_2 W)]^*
\\=&\int_a^b p\, W \left[ W^{-1}\comp(-\partial -(P^{-1}P')^*)\comp(F_2 W)(q)\right]^*
\\=&\int_a^b p\, W [(\partial -P^{-1}P')^{\dagger}\comp F_2 (q)]^*
=\int_a^b p\, W [\A ^{\dagger}\comp F_2(q)]^*
\\
=&\langle p  ,\A ^{\dagger}\comp F_2(q)\rangle_{ W}.
\end{align*}
Hence,
$$
\langle \A(p) ,q\rangle_{\tilde W}
=\langle p  ,\A ^{\dagger}\comp F_2(q)\rangle_{ W}.
$$
In the same way, we can prove that
$$
\langle p ,\A(q)\rangle_{\tilde W}
=\langle \A ^{\dagger}\comp F_2(p),q\rangle_{ W}.
$$
Now it follows easily that
\begin{align*}
\langle  \tilde D(p),q\rangle_{\tilde W}
=&-\langle  \A\comp\A ^{\dagger}\comp F_2 (p) ,q\rangle_{\tilde W}\\
&=-\langle  \A ^{\dagger}\comp F_2 (p) , \A ^{\dagger}\comp F_2(q)\rangle_{W}
\\
&=-\langle p , \A \comp\A ^{\dagger}\comp F_2(q)\rangle_{\tilde W}\\
&
=\langle p  ,\tilde{D}(q)\rangle_{\tilde W}.
\end{align*}

In this connection, it was used the fact that the coefficients of $\A ^{\dagger}$ are rational functions (with no poles in the closure of $(a,b)$), which follows by straightforward computations from \eqref{sadj}.

To finish the proof, we only need to observe that if for any $n$ we have $ D(P_n)=\Gamma_n P_n$, then
\begin{align*}
\langle \A(P_n) ,\A(P_m) \rangle_{\tilde W}
=&\langle P_n ,\A ^{\dagger}\comp F_2\comp\A(P_m)\rangle_{ W}
=\langle P_n ,D(P_m)\rangle_{ W}
=\langle P_n ,\Gamma_mP_m \rangle_{ W}
\\
=&\langle P_n ,P_m \rangle_{ W}\Gamma_m^*=0.
\end{align*}
\end{proof}

The previous Lemmas can be used to explicitly find a signed weight matrix $\W$ (which eventually will be a weight matrix) such that the polynomials $\Pp_n=\A (P_n) $ are orthogonal with respect to $\W$.

\begin{Theo}\label{t10} Let $W$ be a signed weight matrix invertible in the closure of the interval $(a,b)$. Let $D$ be a second order differential operator of the form (\ref{gen-D}) satisfying the symmetry conditions (\ref{locdeq3.1}) and (\ref{locdeq3.2}) and $\lim_{x\to{a,b}}p (F_2 W)^*=0$ for $p\in \Rr_{a,b}$, with $F_2$ invertible on $(a,b)$ and $F_j\in \Rr_{a,b}$.
Assume that we have polynomials $P_n$, $n\in \NN\setminus X$, and $X$ finite, of degree $n$ with nonsingular leading coefficient, orthogonal with respect to $W$ and eigenfunctions of $D$.
Let ${P}$ be a function in the kernel of $D$, invertible on the closure of $(a,b)$, satisfying (\ref{lus}) and for which there exists a polynomial $U$ such that
$(P')^{-1}{P}U$ is a polynomial without zeros in the closure of $(a,b)$, and consider the operators $\A$ and $\B$ defined in Lemmas \ref{lde} and \ref{conta}. Then, we have
\begin{equation}\label{kt7}
\mathcal A=(\partial{P'}^{-1}PU -U) \quad \text{ and }\quad \mathcal B=-(\partial -{P}^{-1}{P}')^\dagger\comp(U^{-1}{P}^{-1}{P}'F_2),
\end{equation}
and the operator
$$
\mathcal D=\mathcal A\comp\mathcal B
$$
is symmetric with respect to the pair $(\W,\Aa)$ where
\begin{equation}\label{exc-w}
\mathcal W=\left(U^{-1}{P}^{-1}{P}'\right) F_2W\left(U^{-1}{P}^{-1}{P}'\right)^*
\end{equation}
and $\Aa$ is the linear subspace of matrix polynomials generated by the polynomials $\Pp_n=\mathcal A(P_n)$, $n\in \NN\setminus X$.
Moreover, the polynomials $\Pp_n$, $n\in \NN\setminus X$ are orthogonal polynomials with respect to $\mathcal W$ and eigenfunctions of $\mathcal D$: $ \D(\Pp_n)=\Gamma_n\Pp_n$.
\end{Theo}

\begin{proof}
Lemma \ref{self-adj-nopol} gives that
\begin{align*}
D=(\partial -{P}^{-1}{P}')^\dagger \comp( U^{-1}{P}^{-1}{P}'(-F_2))\comp(\partial{P'}^{-1}PU -U) =\mathcal B\comp\mathcal A,
\end{align*}
where $\mathcal A$ and $\mathcal B$ (see (\ref{kt7})) are the same operators as in Lemmas \ref{lde} and \ref{conta} for $\Psi=0$).

Also, one has
\begin{align*}
\mathcal D=&\mathcal A\comp\mathcal B =(\partial{P'}^{-1}PU -U)\comp(\partial -{P}^{-1}{P}')^\dagger\comp(U^{-1}{P}^{-1}{P}'(-F_2))
\\
=& \left({P'}^{-1}PU\right)\comp\,\tilde D\, \comp\left(U^{-1}{P}^{-1}{P}'\right)
=\left(U^{-1}{P}^{-1}{P}'\right)^{-1}\comp\,\tilde D\,\comp
 \left(U^{-1}{P}^{-1}{P}'\right) .
\end{align*}
Since
$(\partial -{P}^{-1}{P}')(P_n)$, $n\in \mathbb N\setminus X$,  are orthogonal with respect to ${\tilde W}$, it is immediate to see that
$ \mathcal A(P_n)$, $n\in \mathbb N\setminus X$,  are orthogonal polynomials with respect to $\mathcal W$. Also, by Lemma \ref{lde} one has that $ \mathcal D(\mathcal A(P_n))=\Gamma_n \mathcal A(P_n)$, for every $n\in \mathbb N\setminus X$.

Since $\tilde D$ is symmetric with respect to $(\tilde W,\Rr_{a,b})$, it is completely immediate to see that $\mathcal D$ is symmetric with respect to the pair $(\W,\Aa)$, where
$\mathcal W$ is the (possible signed) weight (\ref{exc-w}) and  $\Aa$ is the linear subspace of matrix polynomials generated by the polynomials $\mathcal A(P_n)$, $n\in \NN\setminus X$.
\end{proof}

\bigskip

We finish this Section pointing out that our guess is that under the hypothesis of Lemma \ref{self-adj-nopol}, the function $P$ in the kernel of $D$ will satisfy the condition (\ref{lus}) under very mild assumptions.
For instance, the following Lemma shows that this is the case when $P$ is a polynomial.

\begin{Lem}\label{self-adj-pol}
Assume $W$ is a determinate signed weight matrix and $D$ is a symmetric second order differential operator with respect to $W$ satisfying the boundary conditions (\ref{boundaryconditions}).
Assume that $P$ is a polynomial in the kernel of $D$ and invertible in the closure of $(a,b)$ satisfying
\begin{equation}\label{bc2}
\lim_{t\to a^+,b^-} t^ng(t)W(t)=0,
\end{equation}
where $g$ is the function (\ref{fug}). Then the condition (\ref{lus}) holds.
\end{Lem}

\begin{proof}
We just sketch the proof. Write $\tilde D$ for the second order differential operator in the right-hand side of (\ref{sadj}).
Performing an integration by parts, using the boundary conditions (\ref{boundaryconditions}) and
(\ref{bc2}), we conclude that $\tilde D$ is symmetric with respect to $W$. Hence, the operator $\tilde D-D$ is also symmetric with respect to $W$.
The identity (\ref{dtd}) in Lemma \ref{self-adj-nopol} gives
$$
\tilde D-D=\partial g-{P}^{-1}{P}' g.
$$
Since $W$ is determinate, the symmetry of $\tilde D-D$, gives the symmetry equations
$$
gW=-Wg^*\quad \text{ and } \quad  {P}^{-1}{P}' g W = (Wg^*)'+W({P}^{-1}{P}'g)^*.
$$
Working with these equations, we have
$$
{P}^{-1}{P}' g W {P}^*= (Wg^*)'{P}^*+W({P}'g)^*=-(gW)'{P}^*-gW({P}')^*.
$$
This gives
$$
({P} g W {P}^*)'={P}' g W {P}^*+{P}(gW)'{P}^*+{P} gW({P}')^*=0
$$
Then, $C= {P} g W {P}^*$ is a constant.
Since ${P}$ is a polynomial, it is immediate from \eqref{bc2} that $C=0$, hence $g=0$. Then $L=0$, which implies that both sides of \eqref{sadj} are equal.
\end{proof}

\begin{Rem}
In this section, we asked the rational functions to be invertible not only in $(a,b)$ but also in its closure. This hypothesis on the (finite) endpoints was a choice to simplify the computations and the explanations; in certain cases, this is not needed (for example, when  $W$  has a zero at that endpoint of degree high enough), and all the results of this section are still valid. The example given in Section \ref{jac-nonpol} illustrates this point.
\end{Rem}

\begin{Rem}
We considered the function $U$ properly chosen such that ${P'}^{-1}PU'$ is a polynomial. This requirement in some very particular cases can be softened to the condition that $ (\partial{P'}^{-1}PU -U)(P_{n})$ is a polynomial for every $n$. The cases in which this consideration is meaningful are extremely rare.
\end{Rem}

\begin{Rem}
If we start with $W$ being a weight matrix (i.e., positive definite), when applying a one-step quasi-Darboux transformation, the seed function $P$ should not be an orthogonal polynomial with respect to $W$, otherwise $\mathcal W$ would have poles inside $(a,b)$. In these situations, one should take $P$ a non-polynomial function. However, this problem of choosing orthogonal polynomials with respect to $W$ may eventually disappear if several quasi-Darboux transformations are successively applied.

If one wants to consider $P$ an orthogonal polynomial with respect to $W$ for a one-step quasi-Darboux transformation, one should take $W$ non-positive definite. Then, one can build an exceptional weight matrix $\W$ (that are positive definite) with \eqref{exc-w}.

In the following Subsections we will show examples of all those scenarios. Moreover, in Subsections \ref{ss1} and \ref{s12} we will construct examples using orthogonal polynomials with respect to a properly signed weight matrix  as seed functions in a one-step quasi-Darboux transformation (Subsection \ref{ss1}), and with respect to a weight matrix (Subsection \ref{s12}) in a two-step quasi-Darboux transformations. In both cases, we also show how to construct the very same examples using non polynomial seed functions.
\end{Rem}

\section{A collection of instructive examples}\label{s1}

\subsection{First example: applying a one-step quasi-Darboux transformation with polynomial seed function to a Hermite type polynomials}\label{ss1}
Our starting point is the following example of orthogonal matrix polynomials of Hermite type. The signed weight matrix is given by $W(t)dt$, where
\begin{equation}\label{wbx}
W(t)=W_{a,\xi}(t)=e^{-t^2}e^{At}\begin{pmatrix}
      \xi &0 \\
      0 & 1
    \end{pmatrix}e^{A^*t},\quad
    A = \left (\begin{matrix} 0 & a \\ 0 & 0 \end{matrix}\right ), \,a\not=0,\,
    \xi\in \RR.
\end{equation}
It was introduced for arbitrary size and $\xi=1$ in \cite{DuGr1}. Only when $\xi>0$, $W(t)$  is properly a weight matrix
(only in that case it is positive definite), but if $\xi\not =0$ there always exists a sequence of orthogonal matrix polynomials $P_{n,\xi}$, $n\ge 0$, with respect to $W_{a,\xi}$, which can be explicitly written in terms of the Hermite polynomials as follows:
\begin{equation}\label{bap}
P_n(t)=P_{n,a,\xi}(t)=H_n(t)\begin{pmatrix}
      1&0 \\
      0 &\xi
    \end{pmatrix}+nH_{n-1}(t)\begin{pmatrix}
      0 &-a \\
      -a & a^2t
    \end{pmatrix}
\end{equation}
(let us remark that $H_n$ stands for the usual $n$-th Hermite polynomial defined by $H_n(t)=(-1)^{n}e^{t^{2}}\frac{d^{n}}{dt^{n}}(e^{-t^{2}})$). The $L^2$-norm of the orthogonal matrix polynomial $P_{n,a,\xi}$ is given by
\begin{equation}\label{norm}
\Vert P_{n,a,\xi}\Vert _{W_{a,\xi}}^2=\sqrt \pi 2^nn!\begin{pmatrix}\xi +a^2(n+1)/2&0\\0&\xi(\xi+a^2n/2)\end{pmatrix}.
\end{equation}
(for $\xi=1$ see \cite{DuGr3}). To simplify the notation we sometimes write $P_n$, $P_{n,\xi}$ or $P_{n,a}$ for the polynomials $P_{n,a,\xi}$ (\ref{bap}).

Actually, for $\xi>0$, up to a change in the parameter $a$ the sequence $(P_{n,\xi})_n$ is equivalent to $(P_{n,1})_n$, in the sense that the weight matrix satisfies the following conjugation
$$
MW_{a,\xi}M^*=W_{a/\sqrt{\xi},1},\quad M=\begin{pmatrix}
      1/\sqrt{\xi}&0 \\
      0 &1
    \end{pmatrix}.
$$
Similarly for $\xi<0$, we have
\begin{equation}\label{mco}
MW_{a,\xi}M^*=W_{a/\sqrt{|\xi|},-1},\quad M=\begin{pmatrix}
      1/\sqrt{|\xi|}&0 \\
      0 &1
    \end{pmatrix}.
\end{equation}
Surprisingly enough, for the construction of exceptional orthogonal matrix polynomials, both cases $\xi>0$ and $\xi<0$ are of similar interest, as we will see below.

The matrix polynomials $P_{n,\xi}$, $n\ge 0$, are eigenfunctions of the following four second-order differential operators (for $\xi=1$ see \cite{CaGr1})
\begin{align*}
D_1&=\partial ^2 I +\partial \begin{pmatrix}
      -2t & 2a \\
      0 & -2t
   \end{pmatrix}+\begin{pmatrix}
      -2 & 0 \\
      0 & 0
   \end{pmatrix},\\
D_2&=\partial ^2 \begin{pmatrix}
      -a^2/4 & a^3t/4 \\
      0 & 0
   \end{pmatrix} +\partial \begin{pmatrix}
      0 & a\xi/2 \\
      -a/2 & a^2t/2
   \end{pmatrix}+\begin{pmatrix}
      0 & 0 \\
      0 & \xi
   \end{pmatrix},\\
D_3&=\partial ^2 \begin{pmatrix}
      -a^2t/2 & a^3t^2/2 \\
      -a/2 & a^2t/2
   \end{pmatrix} +\partial \begin{pmatrix}
      -a^2-\xi & a(a^2+2\xi)t \\
      0 & \xi
   \end{pmatrix}+\begin{pmatrix}
      0 & \xi(a^2+2\xi)/a \\
      0 & 0
   \end{pmatrix},\\
D_4&=\partial ^2 \begin{pmatrix}
      0 & -a^2/4 \\
      0 & 0
   \end{pmatrix} +\partial \begin{pmatrix}
      a/2 & 0 \\
      0 & -a/2
   \end{pmatrix}+\begin{pmatrix}
      0 &0 \\
      1 & 0
   \end{pmatrix}.
\end{align*}

Hence the polynomial $P_{n,\xi}$ is also an eigenfunction of
\begin{equation}\label{sdop}
D_{u_{1},u_{2},u_{3},u_{4},u_{5}}=u_5 I+\sum_{i=1}^4 u_iD_i.
\end{equation}
The operators $D_i$, $i=1,2$, are symmetric with respect to the matrix weight $W_{a,\xi}$, but the operators $D_i$, $i=3,4$ are not. The eigenvalues of $P_{n,\xi}$ with respect to $D_{u_{1},u_{2},u_{3},u_{4},u_{5}}$ are
\begin{equation}\label{eq71}
\Gamma_n=\begin{pmatrix} -2(n+1)u_1+u_5&((n+1)a^2+2\xi)u_3/a\\(na^2/2+\xi)u_4&-2nu_1+(na^2/2+\xi)u_2+u_5\end{pmatrix},\quad n\ge 0.
\end{equation}

\medskip

We next construct our first example of exceptional matrix polynomials. It is constructed by
applying a one-step quasi-Darboux transformation to the matrix polynomials $(P_{n,\xi})_n$ (\ref{bap}). We will start by choosing the polynomial $P_{1,\xi}$ as the seed function, and proceed by displaying all the details.

Consider then the linear space of differential operators of order at most two (\ref{sdop})
\begin{equation}\label{eq3}
D=D_{u_1,u2,u_3,u_4,u_5,\xi}=\partial^2 F_{2}+\partial F_{1}+F_{0},\quad \xi, u_i\in \RR, 1\le i\le 5.
\end{equation}
Let us remind that the matrix polynomials $(P_{n,\xi})_n$ are eigenfunctions of each one of these operators $D$ (with eigenvalues given by (\ref{eq71})).

We next construct a factorization of the operator $D$ as in Lemmas \ref{lde} and \ref{conta}, using the matrix polynomial
$$
P_{1,\xi}(t)=\begin{pmatrix}
      2t & -a \\
      -a & t(a^2+2\xi)
    \protect\end{pmatrix}
$$
as the seed function. Its determinant is
$$
\det(P_{1,\xi})=2t^2(a^2+2\xi )-a^2.
$$

From the polynomial $P_{1,\xi}$, we define the first-order differential operator
\begin{equation}\label{eq4}
\A(y)=yA_0+y'A_1,\quad \mbox{$A_0=-I$ and $A_1=A_1(\xi)=(P'_{1,\xi})^{-1}P_{1,\xi}$}.
\end{equation}

As written above, we will next apply to each operator $D$ (\ref{eq3}) a one-step quasi-Darboux transformation using Lemmas \ref{lde} and \ref{conta}. Because of the choice of the first order differential operator $\A$ (\ref{eq4}), the new second order differential operator $\D$ will have a singularity in the zeros of the polynomial
$\det (P_{1,\xi})$.
This will imply some problem for the existence of a weight matrix with respect to which the exceptional matrix polynomials $\Pp_n=\A (P_{n,\xi})$ are orthogonal. If $\xi>0$, since $P_{n,\xi}$ are orthogonal with respect to a weight matrix, then the polynomial $\det (P_{1,\xi})$ has always real zeros.
However, for some instances of $\xi<0$, the polynomial $\det (P_{1,\xi})$ has non-real zeros.
It is easy to see that $\det (P_{1,\xi})$ has no real zeros if and only if $a^2<-2\xi$. Since the sequence of polynomials $(P_{n,\xi})_n$ are equivalent for all $\xi<0$ (because of the conjugation (\ref{mco})), we just take
\begin{equation}\label{eq4a}
\mbox{$\xi(a)=1-a^2$ and $a^2>2$},
\end{equation}
so that  $\det (P_{1,\xi(a)})=2(2-a^2)t^2-a^2$ has no real zeros. Let us note that then the symmetric matrix function $W_{a,\xi(a)}$ (\ref{wbx}) is not a weight matrix (because it is not positive definite) but a \textit{signed} weight matrix. But, with this choice, the coefficients of the operator $\D$ will be smooth matrix functions in the whole real line. Let us note that now the operators $D$ only depend on $u_i$, $1\le i \le 5$ (and, of course, on $a$ as well):
\begin{equation}\label{eq3a}
D=D_{u_1,u2,u_3,u_4,u_5}=\partial^2 F_{2}+\partial F_{1}+F_{0},\quad u_i\in \RR, 1\le i\le 5.
\end{equation}

Using Lemma \ref{lde}, we can decompose each operator $D$ (\ref{eq3a}) in the form (take into account that $F_i$ depends on $u_i$, $1\le i\le 5$)
$$
D(y)=\B(\A(y))-y\Psi
$$
where (after easy computations)
\begin{align*}
\B (y)&=yB_0+y'B_1,\\
B_0&=A_1^{-1}F_1,\quad B_1=A_1^{-1}F_2,
\end{align*}
and
$$
\Psi=-A_1^{-1}F_1-F_0
$$
is a function of $t$ (see (\ref{eq4}) for the definition of $A_0$ and $A_1$). Since $P_{1,\xi(a)}$ is an eigenfunction of $D$, we get from Lemma \ref{conta},
that the matrix polynomials
\begin{equation}\label{pexct1}
\Pp_{n}(t)=\A(P_{n,\xi(a)}), \quad n\ge 1,\quad \Pp_{0}=\begin{pmatrix} -1&0\\0&a^2-1\end{pmatrix},
\end{equation}
are eigenfunctions of the second-order differential operators
\begin{equation}\label{eq6}
\D(y)=\D_{u_1,u_2,u_3,u_4,u_5}(y)=\A(\B(y))-yA_1^{-1}\Psi A_1,\quad u_i\in \RR, 1\le i\le 5,
\end{equation}
whose coefficients are now matrix rational functions of $t$ (which are smooth when $a^2>2$).

Using (\ref{bap}) and (\ref{eq4}), one can show that the matrix polynomial $\Pp_n$ has degree $n$, except for $n=1$ because from the definition of $\A$, we have $\Pp_1=0$. For $n\not=1$, the leading coefficient is nonsingular and equal to
$$
(n-1)2^{n-1}\begin{pmatrix} 2&0\\0&(n-2)a^2+2\end{pmatrix},\quad n\ge 1.
$$
The eigenvalues can be computed from (\ref{eq71}) by setting $\xi(a)=1-a^2$ (see (\ref{eq4a})).

For $u_1=1$, $u_i=0$, $2\le i\le 5$, the explicit expression of the second-order differential operator $\D$ (\ref{eq6}) is
\begin{align*}
\D (y)&=y\F_0+y'\F_1+y''\F_2,\\
\F_2&=I,\quad \F_1(t)=\frac{1}{\mathfrak p(t)}\begin{pmatrix}
     -2t(\mathfrak p(t)+4) & -4a(a^2-2)(t^2-\frac12) \\
     \frac{8a}{a^2-2} & -2t(p(t)-4a^2+4)
    \protect\end{pmatrix},\quad \F_0=\begin{pmatrix}-2&0\\0&0\end{pmatrix},
\end{align*}
where $\mathfrak p(t)=\det P_{1,\xi(a)}(t)=2(2-a^2)t^2-a^2\not =0$, $t\in \RR$ (assuming $a^2>2$).

For $u_1=1$, $u_2=4/(a^2-2)$, $u_3=u_4=0$ and $u_5=4$, we have that $\Gamma_1=0$. Consider then the operator
$$
D_0=D_{1,4/(a^2-2),0,0,4}.
$$
The coefficient $F_2$ of $\partial ^2$ for this operator is
$$
F_2=\begin{pmatrix} -2/(a^2-2) & a^3t/(a^2-2)\\0 &1\end{pmatrix}.
$$
Since $D_0$ is symmetric for $W_{a,\xi(a)}$ and the seed function $P_{1,\xi(a)}$ satisfies (\ref{lus}), we deduce from Theorem \ref{t10} that the polynomials $\Pp_n$ are orthogonal with respect to the weight matrix
$$
\W_a=P_{1,\xi(a)}^{-1}P_{1,\xi(a)}'F_2W_{a,\xi(a)}(P_{1,\xi(a)}^{-1}P_{1,\xi(a)}')^*,
$$
which indeed it is properly a weight matrix when $a^2>2$ since $\det P_{1,\xi(a)}\not =0$, $t\in \RR$, and the matrix function
$$
F_2(t)W_{a,\xi(a)}(t)=e^{-t^2}\begin{pmatrix}
\frac{a^2(a^2-2)t^2+2(a^2-1)}{a^2-2}
& at\\
    at& 1
    \protect\end{pmatrix}
$$
is positive definite ($\xi(a)=1-a^2$ as in (\ref{eq4a})). The explicit expression of $\W_a$ (after removing the factor $4(a^2-2)$) is
\begin{equation}\label{wex}
\W_a(t)=\frac{e^{-t^2}}{\mathfrak p^2(t)}\begin{pmatrix}
\frac{a^2}{4(a^2-2)}\mathfrak p(t)^2+\frac{\mathfrak p(t)}{a^2-2}-a^2
& a\left(\frac{\mathfrak p(t)}{a^2-2}-2\right)t\\
    a\left(\frac{\mathfrak p(t)}{a^2-2}-2\right)t& \frac{2(a^2(a^2-2)-\mathfrak p(t))}{(a^2-2)^2}
    \protect\end{pmatrix},
\end{equation}
and $\mathfrak p(t)=\det P_{1,\xi(a)}(t)=(4-2a^2)t^2-a^2\not =0$, $t\in \RR$, when $a^2>2$.

We are now ready to prove that the matrix polynomials $\Pp_n$, $n\not =1$, are exceptional matrix polynomials.

\begin{Theo}\label{th6} Assuming $a^2>2$, the matrix polynomials $\Pp_{n}$, $n=0,2,3,\dots$, are orthogonal with respect to the weight matrix $\W_a(t)dt$ (\ref{wex}) and their norm is given by
\begin{equation}\label{normx}
\Vert \Pp_n\Vert _{\W}^2=\frac{n-1}{2(a^2-2)}\begin{pmatrix}1&0\\0&-\frac{2}{a^2-2}\end{pmatrix}\Vert P_{n,\xi(a)}\Vert _{W_{a,\xi(a)}}^2,
\end{equation}
with $\xi(a)=1-a^2$ (see the identity (\ref{norm}) for $\Vert P_{n,\xi(a)}\Vert _{W_{a,\xi(a)}}^2$). Moreover, their linear combinations are dense in  $L^2(\W_a)$.
\end{Theo}

\begin{proof}
We have only to prove the completeness of $\Pp_n$, $n\in \NN\setminus \{1\}$. We proceed in three steps.

Notice first that
\begin{equation}\label{enw}
\Vert f\Vert_{2,\W}=\Vert f/\mathfrak p \Vert_{2,\widetilde W},\quad f\in L^2(\W),
\end{equation}
where $\widetilde W=\mathfrak p \W$.

\bigskip
\noindent
\textsl{Step 1.} For each $r>0$ the linear space $\{(1+t^{2r})p: p\in \PP\}$ is dense in $L^2(\widetilde W)$.

Since $1+t^{2r}>0$, $x\in \RR$, this is equivalent to the density of $\PP$ in $L^2((1+t^{2r})\widetilde W)$. But this follows straightforwardly taking into account that $(1+t^{2r})\widetilde W$ is a determinate weight matrix (see Remark \ref{rdet}).

\bigskip
\noindent
\textsl{Step 2.} $\{\mathfrak p^2 p: p\in \PP\}$ is dense in $L^2(\W)$.

Take a function $f\in L^2(\W)$. Define the function
$g(t)=(1+t^2)f(t)/\mathfrak p^2(t)$. Since $\mathfrak p(t)\not =0$, $t\in\RR$, we deduce that $g\in L^2(\widetilde W)$. Given $\epsilon >0$ and using the first step, we get a polynomial $p$ such that
\begin{equation}\label{dmp}
\Vert g(t)-(1+t^2)p(t))\Vert_{2,\widetilde W}^2<\epsilon.
\end{equation}
Write $\gamma=\inf \{(1+t^2)/\mathfrak p(t),t\in\RR \}$. We then get
\begin{align*}
\Vert g(t)-(1+t^2)p(t))\Vert_{2,\widetilde W}^2&=\left\Vert \frac{1+t^2}{\mathfrak p^2(t)}[f(t)-\mathfrak p^2(t) p(t)]\right\Vert_{2,\widetilde W}^2
\\
&=\left\Vert \frac{1+t^2}{\mathfrak p(t)}[f(t)-\mathfrak p^2(t)p(t)]\right\Vert_{2,\W}^2
\\
&\ge \gamma ^2\left\Vert f(t)-\mathfrak p^2(t)p(t)\right\Vert_{2,\W}^2 .
\end{align*}
Using (\ref{dmp}), we can conclude that the linear space $\{\mathfrak p^2 p: p\in \PP\}$ is dense in $L^2(\W)$.

\bigskip
\noindent
\textsl{Step 3.} $\{\mathfrak p^2 p: p\in \PP\} \subset \{p\in \PP: \D (p)\in \PP\}\subset \Aa$.

The first inclusion follows straightforwardly from the definition of $\D$.

We prove the second inclusion using complete induction on $\deg p$. Indeed, if $\deg p=0$, since $\Pp_0$ is a nonsingular matrix, we have that $p\in \Aa$.

Assume next that $\deg p=k+1$ and $\D (p)\in \PP$, with $k\ge 0$. We first prove that $k>0$. This is because it is easy to check that if $\deg p=1$ then $\D (p)\not \in \PP$. Write $C$ and $A_{k+1}$ for the leading coefficient of $p$ and $\Pp_{k+1}$, respectively (notice that $k+1>1$). Hence the polynomial
$$
q=p-CA_{k+1}^{-1}\Pp_{k+1}
$$
has degree less than $k+1$ and $\D (q)\in \PP$ (in particular, $\deg q\not =1$). Using the induction hypothesis, we deduce that $q\in \Aa$, and hence also $p\in \Aa$.

\bigskip
The completeness of $\Pp_n$, $n\in \NN\setminus \{1\}$, then follows from Steps 2 and 3.
\end{proof}

\bigskip

The exceptional orthogonal matrix polynomials $\Pp_n$ have an associated linear space of dimension $5$ of differential operators of order at most 2, such that they are eigenfunctions for each operator in this space. This phenomenon is not known in the scalar case, where for each family of exceptional polynomials is known essentially only one second-order differential operator with respect to which they are eigenfunctions.

\bigskip

We next show how to construct the exceptional polynomials $\Pp_n$ from the polynomials $P_{n,\xi}$, $\xi=1$, by applying a one-step Darboux transformation with a non-polynomial seed function.
Indeed, consider the weight matrix defined in \eqref{wbx} with $\xi = 1$, together with the sequence of orthogonal polynomials $P_{n,1}(t)$ (\ref{bap}). Let us note that $W_{a,1}$ is now properly a weight matrix because it is positive definite.

For $u_1=2(a^2-1)$, $u_2=4$, $u_3=u_4=0$ and $u_5=-4(a^2-1)$, consider the second-order differential operator
\begin{align*}
\tilde D&=D_{2(a^2-1), 4, 0,0,-4(a^2-1),1}\\& =\partial^{2} \begin{pmatrix}a^{2}-2 & a^{3}t \\ 0 & 2a^{2}-2 \end{pmatrix} + \partial \begin{pmatrix} -t4(a^{2}-1) & 4a^{3}-2a \\ -2a & -t(2a^{2}-4)\end{pmatrix} \\
& \hspace{.5cm} + \begin{pmatrix}-8(a^{2}-1) & 0 \\ 0 & -4(a^{2}-2) \end{pmatrix}.
\end{align*}
It is not difficult to check that the following non-polynomial function $P$ is in the kernel of the operator $\tilde D$ and satisfies (\ref{lus}),
$$
P(t) = \begin{pmatrix} t & \frac{a}{a^{2}-2} \\ -\frac{a}{2} & t \end{pmatrix} W_{a}(t)^{-1} = \begin{pmatrix}-\frac{2te^{t^{2}}}{a^{2}-2} & \frac{a(2t^{2}+1)e^{t^{2}}}{a^{2}-2} \\ -\frac{a(2t^{2}+1)e^{t^{2}}}{2} & \frac{t(2a^{2}t^{2}+a^{2}+2)e^{t^{2}}}{2}
\end{pmatrix}.
$$
Hence, Lemmas \ref{lde}, \ref{conta} and \ref{self-adj-nopol}  give the following factorization for the operator $\tilde D$: for a smooth matrix-valued function $U(t)$ we have $\tilde D(y) = \mathcal{B}(\mathcal{A}(y))$, with
$$\mathcal{A}(y) = y' (\phi'(t))^{-1}\phi(t) U(t) -yU(t), \quad   \mathcal{B}(y) = (\partial - \phi(t)^{-1}\phi(t)')^{\dagger}\left(yU(t)^{-1}F_{2}(t)\right).$$
For the operator defined by $\dagger$ see Definition \ref{edag}.
We take
$$
U(t) = \begin{pmatrix} \frac{4(2a^{2}t^{2}-a^{2}-4t^{2}-2)(a^{2}-1)}{a^{2}-2} & -32at(a^{2}-1)(a^{2}-2) \\ 8at & 8(a^{2}-2)(2a^{2}t^{2}+a^{2}-4t^{2}-2) \end{pmatrix}
$$
and we obtain explicitly
\begin{align*}
        \mathcal{A}&= \partial \begin{pmatrix}\frac{(4a^{4}-8a^{2}+8)t}{a^{2}-2} & (8a^{2}t^{2}-4a^{2}+8)a(a^{2}-2) \\ \frac{4a(a^{2}-1)}{a^{2}-2} & 16(a^{2}-2)(a^{2}-1)t \end{pmatrix} \\
        &\ \quad - \begin{pmatrix} \frac{4(2a^{2}t^{2}-a^{2}-4t^{2}-2)(a^{2}-1)}{a^{2}-2} & -32at(a^{2}-1)(a^{2}-2) \\ 8at & 8(a^{2}-2)(2a^{2}t^{2}+a^{2}-4t^{2}-2) \end{pmatrix}.
\end{align*}

Note that the operator $\mathcal{A}$ increases the degree of polynomials by two. Then, we define the sequence of exceptional polynomials with a gap at $n = 1$ as
\begin{equation}\label{pexct}
    \tilde\Pp_{0,a}(t) = I, \quad \tilde\Pp_{n+2}(t) = \mathcal{A}(P_{n,1}(t)), \quad  \text{for all } n \ge 0.
\end{equation}
According to the Lemmas \ref{lde} and \ref{conta}, the operator $\mathcal{D}(y) = \mathcal{A}(\mathcal{B}(y))$ admits the sequence $\tilde\Pp_{n}$ as an eigenfunction for all $n \neq 1$ (and it does not admit any polynomial of degree $1$ with a non-singular coefficient as an eigenfunction).

It turns out that actually the polynomials $\tilde\Pp_{n}$ and $\Pp_{n}$ are equivalent by means of the following conjugation, for $n\ge 2$
$$
\tilde\Pp_{n}\begin{pmatrix}\sqrt{\frac{a^2-2}{2(a^2-1)}} & 0 \\ 0 & \frac{1}{4\sqrt{2(a^2-2)(a^2-1)}}  \end{pmatrix}
=\begin{pmatrix} \frac{-\sqrt{2(a^2-2)(a^2-1)}}{n-1} & 0 \\ 0 &
\frac{(a^2-2)\sqrt{a^2-2}}{(n-1)\sqrt{2(a^2-1)}}
 \end{pmatrix}\Pp_{n}.
$$

\bigskip
\begin{Rem}\label{re7m}
It turns out that the polynomials $(P_{n,\xi})_n$ (\ref{bap}) are also eigenfunctions of the second order differential operators $D_{u_1,u_4,u_5}=D(u_1,0,0,u_4,u_5)$ (\ref{eq3a}) (which do not depend on $\xi$). Hence, Lemma \ref{conta} implies that the polynomials
$$
\Pp_{n,\xi}=\A(P_{n,\xi})
$$
are eigenfunctions of the second order differential operators
$$
\D_{u_1,u_4,u_5}=\D(u_1,0,0,u_4,u_5)
$$
(\ref{eq6}) (which again do not depend on $\xi$).
The polynomials $\Pp_{n,\xi}$ has degree $n$ with nonsingular leading coefficient, except for $n=1$ with singular leading coefficient.
The eigenfunctions $\Gamma_n$ can be computed from (\ref{eq71}). It is easy to check that the eigenvalues of $\Gamma_n$ are $-2(n+1)u_1+u_5$ and $-2nu_1+u_5$.
Hence, the matrices $\Gamma_n$ and $\Gamma_{m}$ do not share any eigenvalue except when $m=n-1,n,n+1$. Using Lemma \ref{lsyo}, we deduce that
$$
\int_\RR \Pp_{n,\xi}(t)\W (t) \Pp_{m,\xi}^*(t)dt=0,\quad m\not =n-1, n, n+1.
$$
However for $\xi\not=1-a^2$,
$$
\int_\RR \Pp_{n,\xi}(t)\W (t) \Pp_{n\pm 1,\xi}^*(t)dt\not =0.
$$
We do not know if there exists a weight matrix $\widetilde \W$ with respect to which those matrix polynomials are orthogonal.
\end{Rem}

\subsection{Second example: applying a two-step quasi-Darboux transformation with polynomial seed functions to a Hermite type polynomials}\label{s12}
We consider here the weight matrix $W_{a,1}$ (\ref{wbx}) with parameter $\xi=1$, and the orthogonal polynomials $(P_{n,a,1})_n$ (\ref{bap}) with respect to $W_{a,1}$. To simplify the notation, we write $W_a$ and $P_{n,a}$.

The starting point is the linear space of second-order differential operators (\ref{sdop}) for $\xi=1$
\begin{equation}\label{eq11v}
D_u=u_5I+\sum_{i=1}^4u_iD_i,\quad u_i\in \RR, 1\le i\le 5.
\end{equation}
We next construct a factorization of the operator $D_{u}$ using the matrix polynomial
$$
P_{1,a}(t)=\begin{pmatrix}
      2t & -a \\
      -a & t(a^2+2)
    \protect\end{pmatrix},
$$
(see (\ref{bap}), let us remind that we are taking $\xi=1$)
with determinant
$$
\det(P_{1,a})=2t^2(a^2+2 )-a^2.
$$
From the polynomial $P_{1,a}$, we define the first order differential operator
\begin{align}\nonumber
\A_1(y)&=yA_{1,0}+y'A_{1,1},\\\label{eq4z}
A_{1,0}&=-P_{1,a}',\quad A_{1,1}=(P_{1,a}')^{-1}P_{1,a}P_{1,a}'.
\end{align}
We will next apply to the operator $D_{u}$ a one-step quasi-Darboux transformation using Lemmas \ref{lde} and \ref{conta}. Since we have that
$$
D(y)=\B_1 (\A_1 (y))-y \Psi_1,
$$
(where $\B_1$ and $\Psi_1$ depend on the $u$'s and are defined as in Lemma \ref{lde})
we define
\begin{equation}\label{eq10}
\D_{1,u}(y)=\A_1(\B_1(y))-yA_{1,1}^{-1}\Psi_1 A_{1,1}=\partial ^2 \F_{1,2}+\partial \F_{1,1}+\F_{1,0}.
\end{equation}
We also define the polynomials $\Pp_{1,n,a}=\A_1(P_{n,a})$. These polynomials have degree $n$ with non-singular leading coefficient, except for $n=1$, because due to the construction $\Pp_{1,1,a}=0$. A simple computation shows that the leading coefficient is given by
$$
(n-1)2^{n-1}\begin{pmatrix} 4&0\\ 0&(a^2+2)(na^2+2)\end{pmatrix}.
$$
They are also eigenfunctions of the second-order differential operators
$\D_{1,u}$ with eigenvalues given by (\ref{eq71}), with $\xi=1$.
Since in the entries of $\F_{1,2}$ and $\F_{1,1}$ are rational functions with $\det P_{1,a}(t)=2(a^2+2)t^2-a^2$ and $\det^2 P_{1,a}(t)$ in the denominators, respectively, each differential operator $\D_{1,u}$ has a singularity at the zeros of $\det P_{1,a}(t)=2(a^2+2)t^2-a^2$, with are always real for all real values of $a$.

For $u_1=1$ and $u_i=0$, $1\le i\le 4$, the second order differential operator $\D_{1,u}$ has the explicit expression
\begin{align*}
\D_{1,u}&=\partial ^2 \F_{1,2}+\partial \F_{1,1}+F_{1,0},\\
\F_{1,2}&=I,\quad \F_{1,1}=\frac{1}{\det P_{1,a}(t)}\begin{pmatrix} -2t(\det P_{1,a}(t)+4(a^2+1)& a(a^2+2)^2(2t^2-1)\\-16a(a^2+1)/(a^2+2)^2&-2t(\det P_{1,a}(t)+4)\end{pmatrix},\\
\F_{1,0}&=\begin{pmatrix}-2&0\\0&0\end{pmatrix}.
\end{align*}

\bigskip

We next apply a quasi-Darboux transformation, factorizing each second-order differential operator $\D_{1,u}$ (\ref{eq10}). In order to do that, we define the first-order differential operator
$$
\A_2(y)=yA_{2,0}+y'A_{2,1},
$$
where
\begin{equation}\label{eq4z1}
A_{2,0}=-\begin{pmatrix}1/2&0\\0&1/(a^2+2)\end{pmatrix},\quad A_{2,1}=-(\Pp_{1,2,a}')^{-1}\Pp_{1,2,a} A_{2,0}.
\end{equation}
The explicit expression for $A_{2,1}$ is
$$
A_{2,1}=\frac{1}{\det P_{1,a}(t)}\begin{pmatrix}\frac{t(2t^2+1)(a^2+2)}{4}&-\frac{a(t^2-1/2)(a^2+2)}{4}\\-\frac{a(2a^2t^2-3a^2+4t^2-2)}{2(a^2+2)^2}
&\frac{t(2a^2t^2-a^2+4t^2+2)}{2a^2+4}\end{pmatrix}.
$$
Using Lemmas \ref{lde} and \ref{conta}, we can factorize the second order differential operator $\D_{1,u}$ (\ref{eq10}) in the form
$$
\D_{1,u}(y)=\B_2(\A_2(y))-y\Psi_2
$$
where
$$
\B_2(y)=yB_{2,0}+y'B_{2,1},
$$
with
$$
B_{2,0}=A_{2,1}^{-1}(\F_{1,1}-(A_{2,1}'+A_{2,0})B_{2,1})),\quad B_{2,1}=A_{2,1}^{-1}\F_{1,2},
$$
and
$$
\Psi_2(t)=\B_2(A_{2,0})-\F_{1,0}.
$$
This gives the second-order differential operator
\begin{equation}\label{eq11}
\D_{2,u}(y)=\A_2(\B_2(y))-yA_{1,1}^{-1} \Psi_2 A_{1,1}=\partial ^2 \F_{2,2}+\partial \F_{2,1}+\F_{2,0}.
\end{equation}

Using Lemma \ref{conta}, we deduce that the functions
\begin{equation}\label{p 2 step}
\Pp_{2,n,a}(t)=\A_2(\Pp_{1,n,a})
\end{equation}
are eigenfunctions of the second order differential operator $\D_{2,u}$.

The polynomial that appears dividing in the entries of the coefficients of $\D_{2,u}$ is
$$
\det (\Pp_{1,2,a}(t))=8(a^2+1)(4(a^2+2)t^4+8t^2+3a^2+2),
$$
which has no real zeros for any real value of $a$.

We next prove that the function $\Pp_{2,n,a}$ is actually a polynomial of degree $n$ with non singular leading coefficient, except for $n=1,2$, because due to the construction we have $\Pp_{2,2,a}=\Pp_{2,1,a}=0$.

\begin{Lem} Let $P$ be a matrix polynomial of degree $s\not =1,2$ with non-singular leading coefficient $L$. Then $\A_2(\A_1(P))$ is again a
matrix polynomial of degree $s$ with non-singular leading coefficient equal to $(s-1)(s/2-1)L$.
\end{Lem}

\begin{proof}
It is just a matter of computation. Since (see (\ref{eq4z}))
\begin{equation}\label{eq14}
A_{1,0}=-P_{1,a}',\quad A_{1,1}'=P_{1,a}',
\end{equation}
we have
\begin{equation}\label{eq12}
[\A_1(P)]'=(PA_{1,0}+P'A_{1,1})'=P''A_{1,1}.
\end{equation}
We also have using (\ref{eq4z1}) and (\ref{eq12})
\begin{align}\nonumber
A_{2,1}&=-(\Pp_{1,2}')^{-1}\Pp_{1,2}A_{2,0}=-([\A_1(P_2)]')^{-1}\A_1(P_2)A_{2,0}\\\label{eq13}
&=-(P_2''A_{1,1})^{-1}(P_2A_{1,0}+P_2'A_{1,1})A_{2,0}.
\end{align}
We finally have using (\ref{eq12}) and (\ref{eq13})
\begin{align*}
\A_2(\A_1(P))&=\A_1(P)A_{2,0}+(\A_1(P))'A_{2,1}\\
&=[PA_{1,0}+P'A_{1,1}-P''(P_2'')^{-1}(P_2A_{1,0}+P_2'A_{1,1})]A_{2,0}.
\end{align*}
Since $\deg A_{1,0}=0$ and $\deg A_{1,1}=1$ (see (\ref{eq4z})), this shows that $\A_2(\A_1(P))$ is a polynomial of degree at most $s$. Using (\ref{eq14}) and taking into account that $-P_1'=A_{2,0}^{-1}$, we deduce that the  leading coefficient
of $\A_2(\A_1(P))$ is
$$
[L(-P_1'+sP_1'-\frac{s(s-1)}2(-P_1'+2P_1')]A_{2,0}=(s-1)(s/2-1)L.
$$
This completes the proof of the Lemma.
\end{proof}

\medskip

We next show how to construct the polynomials $\Pp_{2,n,a}$ from the polynomial $P_{n,a\sqrt{2}/\sqrt{3a^2+2}}$ (\ref{bap}) in just a one-step Darboux transformation using a non-polynomial seed function (let us remaind that we are taking $\xi=1$ in (\ref{bap})). This is in agreement with what happens with Hermite scalar valued exceptional polynomials except for the subtlety that we have to change the parameter $a$: $a\to a\sqrt{2}/\sqrt{3a^2+2}$.

For
$$
u_1=\frac{3a^2-2}{a^2-1}, \quad u_2=\frac{4}{a^2-1},\quad u_3=u_4=0,\quad u_5=-\frac{4(3a^2-2)}{a^2-1},
$$
consider the second-order differential operator
\begin{align*}
\tilde D&=D_{\frac{3a^2-2}{a^2-1},\frac{4}{a^2-1}, 0,0,-\frac{4(3a^2-2)}{a^2-1},1}\\&
=\partial^{2} \begin{pmatrix}2 & \frac{a^{3}t}{a^{2}-1} \\ 0 & \frac{3a^{2}-2}{a^{2}-1}\end{pmatrix} + \partial \begin{pmatrix} - \frac{2(3a^{2}-2)t}{a^{2}-1} & \frac{2a(3a^{2}-1)}{a^{2}-1} \\ - \frac{2a}{a^{2}-1} & -4t \end{pmatrix} + \begin{pmatrix}- \frac{6(3a^{2}-2)}{a^{2}-1} & 0 \\0 & -12 \end{pmatrix}.
\end{align*}
Denote by $F_{2,a}$ for the coefficient of the second derivative in $\tilde D$:
\begin{equation}\label{f2a}
F_{2,a}(t)= \begin{pmatrix}2 & \frac{a^{3}t}{a^{2}-1} \\ 0 & \frac{3a^{2}-2}{a^{2}-1}\end{pmatrix} .
\end{equation}

It is not difficult to check that the following non-polynomial function $P$ is in the kernel of the operator $\tilde D$,
$$
P(t) = \begin{pmatrix}\frac{2a^{2}t^{2}-2t^{2}-1}{2(a^{2}-1)} & \frac{at}{a^{2}-1} \\ -at & t^{2} + \frac{1}{2} \end{pmatrix}W_{a}(t)^{-1} = e^{t^{2}}\begin{pmatrix}-\frac{2t^{2}+1}{2(a^{2}-1)} & \frac{at(2t^{2}+3)}{2(a^{2}-1)} \\ - \frac{at(2t^{2}+3)}{2} & \frac{(2a^{2}t^{4}+3a^{2}t^{2}+2t^{2}+1)}{2} \end{pmatrix}.
$$
$P$ also satisfies the condition (\ref{lus}).

Hence, Lemmas \ref{lde}, \ref{conta} and \ref{self-adj-nopol} give the following factorization for the operator $\tilde D$: for any smooth matrix-valued function $U(t)$ we have $\tilde D(y) = \mathcal{B}(\mathcal{A}(y))$, with
$$
\mathcal{A}(y) = y' (P'(t))^{-1}P(t) U(t) -yU(t),  \quad \mathcal{B}(y) = (\partial - \phi(t)^{-1}\phi(t)')^{\dagger}\left(yU(t)^{-1}F_{2}(t)\right).
$$
For the operator defined by $\dagger$ see Definition \ref{edag}.
We take
$$
U(t) = \begin{pmatrix} \frac{t(2a^{2}t^{2}-2t^{2}-3)}{2(a^{2}-1)} & - \frac{3a(3a^{2}-2)(2t^{2}+1)}{8(a^{2}-1)} \\ \frac{3a(2t^{2}+1)}{2(3a^{2}-2)} & \frac{t(2t^{2}+3)}{2}\end{pmatrix},
$$
and we obtain explicitly $\A=\partial A_{1,a}+A_{0,a}$, where
\begin{align*}
        \mathcal{A}&= \partial \begin{pmatrix} \frac{(3a^{4}-4a^{2}+2)t^{2}-a^{2}+1}{2(3a^{2}-2)(a^{2}-1)} & \frac{at(2a^{2}t^{2}-3a^{2}+4)}{8(a^{2}-1)} \\ \frac{at}{2(a^{2}-1)} & \frac{(2t^{2}+1)(3a^{2}-2)}{8(a^{2}-1)} \end{pmatrix} \\
        &\hspace{1cm} - \begin{pmatrix} \frac{t(2a^{2}t^{2}-2t^{2}-3)}{2(a^{2}-1)} & - \frac{3a(2t^{2}+1)(3a^{2}-2)}{8(a^{2}-1)} \\ \frac{3a(2t^{2}+1)}{2(3a^{2}-2)} & \frac{t(2t^{2}+3)}{2}\end{pmatrix}.
\end{align*}
Note that the operator $\mathcal{A}$ increases the degree of polynomials by $3$. Then, we define the sequence of exceptional polynomials with gaps at $n =1,2$ given by
\begin{equation}\label{pexct2}
    \tilde\Pp_{0,a}(t) = I, \quad \tilde\Pp_{n+3,a}(t) = \mathcal{A}(P_{n,a}(t)), \quad  \text{for all } n \ge 0.
\end{equation}
According to the Lemmas \ref{lde} and \ref{conta}, the operator $\mathcal{D}(y) = \mathcal{A}(\mathcal{B}(y))$ admits each polynomial $\tilde\Pp_{n,a}$ as eigenfunction for all $n \neq 1,2$ (and it does not admit
any polynomial of degree $1$ or $2$ with a non-singular coefficient as an eigenfunction).

It turns out that actually the polynomials $\tilde\Pp_{n,a\sqrt{2}/\sqrt{3a^2+2}}$ and $\Pp_{2,n,a}$ are equivalent by means of the following conjugation, for $n\ge 2$
$$
\begin{pmatrix} 1 & 0 \\ 0 &
\frac{3a^2+2}{na^2+2} \end{pmatrix}
M\tilde\Pp_{n,\frac{a\sqrt{2}}{\sqrt{3a^2+2}}}
=-\frac{1}{4(n-1)(n-2)}\begin{pmatrix} 1 & 0 \\ 0 &
\frac{1}{na^2/2+1}
 \end{pmatrix}\Pp_{2,n,a}M,
$$
where
$$
M=\begin{pmatrix} \frac{\sqrt{6a^2+4}}{2} & 0 \\ 0 &
\frac{2}{(a^2+2)} \end{pmatrix}.
$$
According to Theorem \ref{t10}, the polynomials $\Pp_{2,n,a}$ are orthogonal with respect to the weight matrix
$$
\W_a=MA_{1,\frac{a\sqrt{2}}{\sqrt{3a^2+2}}}^{-1}F_{2,\frac{a\sqrt{2}}{\sqrt{3a^2+2}}}W_{\frac{a\sqrt{2}}{\sqrt{3a^2+2}}}(MA_{1,a\frac{a\sqrt{2}}{\sqrt{3a^2+2}}}^{-1})^*.
$$
The explicit expression of $\W_a$ is (after
removing the factor $(a^2+2)/(64(3a^2+2)^2)$)
\begin{equation}\label{w2s}
\frac{e^{-t^{2}}}{\mathfrak p^{2}(t)}\begin{pmatrix} \frac{(a^2+2)[(a^{2}t^{2}+(3a^{2}+1))\mathfrak p(t)+a^{2}(3a^{2}+2)(2t^{2}-3)]}{(3a^2+2)^2} & \frac{at(\mathfrak p(t)+8((a^{2}+2)t^{2}+1))}{3a^{2}+2} \\ \frac{at(\mathfrak p(t)+8((a^{2}+2)t^{2}+1))}{3a^{2}+2}  & \frac{\mathfrak p(t) + 2a^{2}\left(\frac{6(a^{2}+2)t^{2}}{3a^{2}+2}-1\right)}{(a^2+2)}\end{pmatrix},
\end{equation}
where $\mathfrak p(t)=4(a^2+2)t^4+8t^2+3a^2+2\not =0$, $t\in \RR$, for all real number $a$.

We have then the following Theorem (the proof is similar to that of Theorem \ref{th6} and it is omitted).

\begin{Theo}\label{th7} For $a\in \RR$, the matrix polynomials $\Pp_{2,n,a}$, $n=0,3,4\dots$, are orthogonal with respect the weight matrix $\W$ with norm
\begin{equation}\label{normx2}
\Vert \Pp_{2,n,a}\Vert _{\W}^2=(n-1)(n-2)\begin{pmatrix}\frac{1}{4(3a^2+2)^2}&0\\0&\frac{1}{4(3a^2+2)(a^2+2)}\end{pmatrix}\Vert P_{n,a}\Vert _{W_{a}}^2,
\end{equation}
where $\Vert P_{n,a}\Vert _{W_{a}}^2$ is given by (\ref{norm}) (for $\xi=1$).
Moreover their linear combinations are dense in  $L^2(\W)$.
\end{Theo}

Actually, the weight matrix can be also computed as follows. We find two second order differential operators $D_{\tilde u}$ (\ref{eq11v}) and $\D_{1,\hat u}$ (\ref{eq10}) such that
$\Gamma_1=0$ for $D_{\tilde u}$ and $\Gamma_2=0$ for $\D_{1,\hat u}$. Indeed, a simple computation shows that for $\tilde u_1=1$, $\tilde u_2=-4/(2+a^2)$, $\tilde u_3=\tilde u_4=0$ and $\tilde u_5=4$, we have that $\Gamma_1=0$. The coefficient $F_2$ of $\partial ^2$ for the operator
$$
D_{\tilde u}=D_{1,-4/(2+a^2),0,0,4}
$$
is
$$
F_2(t)=\begin{pmatrix}1+a^2/(a^2+2) & -a^3t/(a^2+2)\\0 &1\end{pmatrix}.
$$
Similarly, for $\hat u_1=1$, $\hat u_2=-2/(1+a^2)$, $\hat u_3=\hat u_4=0$ and $\hat u_5=6$, we have that $\Gamma_2=0$. The coefficient $\F_{1,2}$ of $\partial ^2$ for the operator
$$
\D_{1,\hat u}=\D_{1,-2/(1+a^2),0,0,6}
$$
is
$$
\F_{1,2}(t)=\begin{pmatrix}\frac{(4t^2-1)a^2+4t^2}{(2t^2-1)a^2+4t^2} & -\frac{a^3(a^2+2)^2t(t^2+1/2)}{(4(a^2+1))((t^2-1/2)a^2+2t^2)}\\
\frac{4a^3t}{(a^2+2)^2(2a^2t^2-a^2+4t^2)} &\frac{(2t^2-3)a^4+(12t^2-2)a^2+8t^2}{4(a^2+1)((t^2-1/2)a^2+2t^2)}\end{pmatrix}.
$$
Then, a simple computation shows that the weight matrix (\ref{w2s}) is, up to a multiplicative constant, equal to
$$
A_{2,1}^{-1}\F_{1,2}A_{1,1}^{-1}F_2W(A_{1,1}^{-1})^*(A_{2,1}^{-1})^*
$$

\subsection{Third example: applying a one-step quasi-Darboux transformation with non-polynomial seed function to a Laguerre type polynomials}
\label{lag}

Let $a \not= 0$, $\alpha > -1$, we consider the matrix-valued Laguerre weight (introduced in \cite{CMV} and for arbitrary size in \cite{DuDI1})
$$
W_{a,\alpha}(t) = e^{-t}t^{\alpha}\begin{pmatrix} t + a^{2}t^{2} & at \\at & 1\end{pmatrix}.
$$
We have a sequence of orthogonal polynomials for $W_{a,\alpha}$ given by
$$
P_{n}(t) = \begin{pmatrix} \ell_{n}^{(\alpha+1)}(t) & a(\ell_{n+1}^{(\alpha)}-\ell_{n}^{(\alpha+1)}(t)t) \\ -an \ell_{n-1}^{(\alpha+1)}(t) & a^{2}n\ell_{n-1}^{(\alpha+1)}(t)t + \ell_{n}^{(\alpha)}(t) \end{pmatrix},
$$
where $\ell_{n}^{(\alpha)}(t)$ is the $n$-th monic Laguerre polynomial of parameter $\alpha$.

The polynomials $P_{n}(t)$ are eigenfunctions of a linear space $\Upsilon$ of (symmetric) differential operators of order at most $2$ of the form
\begin{equation}\label{opel}
D=u_0I+u_1D_1+u_2D_2=\partial^2 F_{2}+\partial F_{1}+F_{0},\quad u_0, u_1, u_2\in \RR,
\end{equation}
where
\begin{align*}
D_1&= \partial^{2}tI + \partial \begin{pmatrix} \alpha +2 -t & at \\ 0 & \alpha + 1 -t \end{pmatrix} + \begin{pmatrix} - \alpha - 3 & a(\alpha+1) \\ 0 & -\alpha - 2 \end{pmatrix},\\
D_2&=\partial^{2} \begin{pmatrix} t & -at^2 \\ 0 & 0 \end{pmatrix} + \partial \begin{pmatrix} \alpha +2  &-\frac{1}{a}(1+a^2(\alpha+2))t \\ \frac1a & -t \end{pmatrix} + \begin{pmatrix} \frac{1}{a^2} & -\frac{\alpha+1}a \\ 0 & 0 \end{pmatrix}.
\end{align*}
Since $W_{a,\alpha}$ is a weight matrix supported in $[0,+\infty)$, the polynomials $P_n$ have all their zeros in $(0,+\infty)$. Hence, if one takes any $P_n$ as a seed function to apply the quasi-Darboux transformation, the generated exceptional second-order differential operators will have singularities in $(0,+\infty)$. This is saying that we likely cannot generate exceptional matrix polynomials using the orthogonal polynomial $P_n$ as a seed function in a one-step quasi-Darboux transformation.
We can avoid this problem by using a non-polynomial function. Indeed, consider the function
$$
P(t) = e^{t}\begin{pmatrix} -(t+\alpha+2) & (t+\alpha+2)at \\ 0 & -(\alpha +1 +t) \end{pmatrix}.
$$
It is easy to check that $P$ is an eigenfunction of each operator $D\in \Upsilon$ (\ref{opel}) with eigenvalue
$$
\Gamma_D=\begin{pmatrix}(\alpha +2)u_2+u_0+\frac{u_2}{a^2}  & -\frac{(\alpha+2)u_2}a \\ \frac{u_2}a & u_0 \end{pmatrix}.
$$
Let us note that for this function
$$
\det P(t)=e^{2t}(t+\alpha+2)(t+\alpha+1)\not =0,\quad t>0.
$$
From the function $P$, we define the first-order differential operator
$$
\A(y)=yA_0+y'A_1,\quad \mbox{$A_0(t)=-U(t)$ and $A_1(t)=P'(t)^{-1}P(t)U(t)$}
$$
where
$$
U(t)=\begin{pmatrix} t + \alpha + 3 & -a(t-2) \\ 0 & t + \alpha + 2 \end{pmatrix}.
$$
A simple computation shows
\begin{equation}\label{eq4l}
\mathcal{A} = \partial \begin{pmatrix}t + \alpha + 2 & a(\alpha+2) \\ 0 & \alpha + 1 + t \end{pmatrix} - \begin{pmatrix} t + \alpha + 3 & -a(t-2) \\ 0 & t + \alpha + 2 \end{pmatrix}.
\end{equation}
Hence, using Lemmas \ref{lde} and \ref{conta}, we can decompose each operator $D\in \Upsilon$ (\ref{opel}) in the form
\begin{equation}\label{eq6ll}
D(y) =\B(\A(y))-y\Psi
\end{equation}
where
\begin{align*}
\B (y)&=yB_0+y'B_1,\\
B_0&=A_1^{-1}(F_1-(A_1'+A_0)A_1^{-1}F_2),\quad B_1=A_1^{-1}F_2,
\end{align*}
and
$$
\Psi=A_0'A_1^{-1}F_2+A_0B_0-F_0,
$$
is a function of $t$. Since $P$ is an eigenfunction of $D$, we get from Lemma \ref{conta}
that the sequence of matrix polynomials with a gap at $n = 0$ (and non-singular leading coefficient)
$$
\Pp_{n}(t)=\A(P_{n-1}), \quad n\ge 1,
$$
are eigenfunctions of the second-order differential operators (see (\ref{eq6ll}))
\begin{equation}\label{eq6l}
\D(y)=\A(\B(y))-yA_1^{-1}\Psi A_1,\quad u_i\in \RR, 0\le i\le 3,
\end{equation}
whose coefficients are now matrix rational functions of $t$ which are smooth in $(0,+\infty)$ for $\alpha>-1$ and  $a\not =0$.

If we take in (\ref{opel}) $u_0=u_2=0$ and $u_1=1$, that is $D=D_1$, it follows that the eigenvalue $\Gamma_{D_1}=0$. It turns out that $P$ also satisfies the identity (\ref{lus}).
Hence, using Lemma \ref{self-adj-nopol}, we can factorize $D_1$ as $D_1 (y)= \mathcal{B}(\mathcal{A}(y))$ with
\begin{equation}\label{AB Lag}
    \begin{split}
        \mathcal{A} &  = \partial \begin{pmatrix}t + \alpha + 2 & a(\alpha+2) \\ 0 & \alpha + 1 + t \end{pmatrix} - \begin{pmatrix} t + \alpha + 3 & -a(t-2) \\ 0 & t + \alpha + 2 \end{pmatrix}, \\
        \mathcal{B} &  = \partial \begin{pmatrix} \frac{t}{t+\alpha+2} & - \frac{ at(\alpha+2)}{(t+\alpha+2)(t+\alpha+1)} \\0 & \frac{t}{t+\alpha+1} \end{pmatrix} + \begin{pmatrix} \frac{\alpha+2}{t+\alpha+2} & - \frac{a(\alpha^{2}+3\alpha-t+2)}{(t+\alpha+2)(t+\alpha+1)} \\ 0 & \frac{\alpha+1}{t+\alpha+1} \end{pmatrix}.
    \end{split}
\end{equation}
This factorization produces the exceptional operator $\mathcal{D}(y) = \mathcal{A}(\mathcal{B}(y))$, given by
\begin{equation*}
    \begin{split}
        \mathcal{D} &  = \partial^{2} t I + \partial \begin{pmatrix} \frac{(t+\alpha+3)(\alpha+2-t)}{t+\alpha+2} & \frac{a(t(\alpha+t)^{2} + \alpha^{2} + 5\alpha t + 2t^{2} + 3\alpha + 7t + 2)}{(t+\alpha+2)(t+\alpha+1)} \\ 0 & \frac{(t+\alpha+2)(\alpha+1-t)}{\alpha+1+t} \end{pmatrix} \\
        & \quad + \begin{pmatrix}- \frac{(\alpha+2)(t+\alpha+4)}{t+\alpha+2} & \frac{a(\alpha(t+\alpha)^{2} + 4\alpha^{2} + 3\alpha t + t^{2} + 5\alpha -3t + 2)}{(t+\alpha+2)(t+\alpha+1)} \\ 0 & - \frac{(\alpha+1)(t+\alpha+3)}{t+\alpha+1} \end{pmatrix}.
    \end{split}
\end{equation*}
Since for $D_1$ the coefficient of the second derivative is $F_2(t)=tI$, we have that the matrix polynomials $\Pp_n$, $n\ge 1$, are orthogonal with respect to the weight matrix
$$
(P U)^{-1}P 'F_2W_{a,\alpha}((P U)^{-1}P ')^*
$$
(see Theorem \ref{t10}). A simple computation shows
$$
\mathcal{W}_{a,\alpha}(t) = e^{-t}t^{\alpha} \begin{pmatrix}\frac{t(a^{2}(t+\alpha+2)^{2}(t-1)^{2} + t(\alpha+1+t)^{2})}{(t+\alpha+2)^{2}(t+\alpha+1)^{2}} & \frac{a(t-1)t}{(t+\alpha+1)^{2}} \\ \frac{a(t-1)t}{(t+\alpha+1)^{2}} & \frac{t}{(t+\alpha+1)^{2}}\end{pmatrix}.
$$

We conclude this section by pointing out that the exceptional sequence $\mathcal{P}_{n}(t)$ is an eigenfunction of a differential operator of odd order. This phenomenon does not occur in the scalar case, where the differential operators having orthogonal polynomials as eigenfunctions are necessarily of even order. The existence of such an operator in the matrix setting follows from the fact that the original sequence $P_{n}(t)$, from which $\mathcal{P}_{n}(t)$ is constructed via quasi-Darboux transformation, is an eigenfunction of the following third-order differential operator:

{\tiny \begin{equation*}
    \begin{split}
        D_{3} &= \partial^{3} \begin{pmatrix}  at^{2} && -t^{2}(a^{2}t+1) \\ t &&  -at^{2} \end{pmatrix} + \partial^{2} \begin{pmatrix}  at(\alpha+5)+\frac{2t}{a} && -a^{2}t^{2}(\alpha+5)-t(2\alpha+t+4) \\ 2 + \alpha && -at(\alpha+2)-\frac{2t}{a} \end{pmatrix} \\
         & \quad + \partial  \begin{pmatrix} \frac{2(a^{2}+1)(2+\alpha)-t}{a} && -\frac{t}{a^{2}} - 2(a^{2}(\alpha+2)+1)t-(\alpha+2)(\alpha+1) \\ \frac{1}{a^{2}} && \frac{t-2(\alpha+1)}{a} \end{pmatrix} + \begin{pmatrix} -\frac{\alpha+1}{a} && \frac{(\alpha+1)(a^{2}\alpha-1)}{a^{2}} \\ -\frac{1}{a^{2}} && \frac{\alpha+1}{a} \end{pmatrix}.
    \end{split}
\end{equation*}
}
Using the operators $\mathcal{A},\mathcal{B}$ defined in \eqref{AB Lag}, we obtain the fifth-order differential operator $\mathcal{D}_{5}(y) = \mathcal{A}(\D_{3}(\mathcal{B}(y)))$. It follows that
$$\mathcal{D}_{5}(\mathcal{P}_{n}(t)) = \begin{pmatrix} 0 & \frac{(\alpha+n+3)(\alpha+n+1)((n+1)a^{2}+1)}{a^{2}} \\ \frac{(na^{2}+1)(\alpha+n+2)}{a^{2}} & 0 \end{pmatrix} \mathcal{P}_{n}(t).$$

\subsection{Fourth example: applying a one-step quasi-Darboux transformation with non-polynomial seed function to a Gegenbauer type polynomials}\label{jac-nonpol}
Let $r > 0$, $0 < a < r$.  We consider the Gegenbauer type weight matrix
$$W_{a,r}(t) = (1-t^{2})^{\frac{r}{2}-1}\begin{pmatrix} a(t^{2}-1)+r & -rt \\ -rt & (r-a)(t^{2}-1)+r \end{pmatrix}$$
which arises in group representation theory and has been studied in detail in \cite{PZ,Z}.
A sequence of orthogonal polynomials for $W_{a,r}$ is given by
$$
P_{n}(t) = \begin{pmatrix} p_{n}(t)' & (r-a)p_{n}(t)+p_{n}(t)'t \\ ap_{n}(t)+p_{n}(t)'t & p_{n}(t)'\end{pmatrix},
$$
where $p_{n}(t)$ is the $n$-th Jacobi polynomial of parameters $\alpha = \beta = \frac{r}{2}$.

The polynomials $P_{n}(t)$ are eigenfunctions of a $4$-dimensional linear space $\Upsilon$ of differential operators of order at most $2$. The following second order differential operator $\tilde D\in \Upsilon$, and moreover, $\tilde D$ is symmetric with respect to $W_{a,r}$:
{\footnotesize \begin{equation}\label{tdj}
    \begin{split}
        \tilde D & = \partial^{2} \begin{pmatrix} (1-a)(2+a-r)t^{2} + (a-2)(a-r+1) & t(r-2a) \\ t(2a-r) & (2-a)(a+1-r)t^{2}+(a-1)(a-r+2)\end{pmatrix} \\
        & \quad + \partial \begin{pmatrix} t(1-a)(2+a-r)(r+2) & r^{2}-ar-2a-2r+4 \\ ar+2a-4r+4 & t(2-a)(a-r+1)(r+2) \end{pmatrix} \\
        &\quad + \begin{pmatrix}2(r-1)(1-a)(2-r+a) & 0 \\ 0 & 2(r-1)(2-a)(a-r+1) \end{pmatrix}.
    \end{split}
\end{equation}
}
As it happens with the previous example, since $W_{a,r}$, $r > 0$, $0 < a < r$,  is a weight matrix supported in $[-1,1]$, the polynomials $P_n$ have all their zeros in $(-1,1)$. Hence, it is not a good idea to take an orthogonal polynomial $P_n$ as a seed function to apply a one-step quasi-Darboux transformation.
We can avoid this problem by using the following non-polynomial function ($a\not =1, r-1$)
\begin{align*}
P(t)& = \begin{pmatrix} t & - \frac{1}{a-r+1} \\ \frac{1}{a-1} & t\end{pmatrix} W_{a,r}(t)^{-1} \\
& = (1-t^{2})^{-\frac{r}{2}-1}\begin{pmatrix} \frac{t((a-r+1)t^{2}-a-1)}{a(a-r+1)} & - \frac{rt^{2}-t^{2}+1}{a(a-r+1)} \\ - \frac{rt^{2}-t^{2}+1}{(a-r)(a-1)} & - \frac{t(at^{2}-t^{2}-a+r+1)}{(a-r)(a-1)}\end{pmatrix}.
\end{align*}
It turns out that $P$ is an eigenfunction of each operator $D\in \Upsilon$. From the function $P$, we define the first order differential operator
$$
\A(y)=yA_0+y'A_1,\quad \mbox{$A_0(t)=-U(t)$ and $A_1(t)=P'(t)^{-1}P(t)U(t)$}
$$
where ($a\not =2, r-2$)
$$
U(t)=\begin{pmatrix}(1-r)t^{2}+\frac{r+1-a}{a+1-r} & \frac{2t(1-r)(a+2-r)}{(a-2)(a+1-r)} \\ \frac{2t(r-1)(a-2)}{(a-1)(a+2-r)} & - \frac{t^{2}(ar-a-r+1)+a+1}{a-1} \end{pmatrix}.
$$
A simple computation shows
\begin{equation}\label{eq4j}
   \begin{split}
\mathcal{A} &= \partial\begin{pmatrix}t^{3} + \frac{(-a^{3}+2a^{2}r-ar^{2}-2ar+2r^{2}+a-3r+2)t}{(a-1)(2+a-r)(a-r+1)} & \frac{t^{2}(a-1)(3a-2r+2)-(a-2)(a-r+1)}{(a-1)(a-2)(a-r+1)} \\ - \frac{t^{2}(3a-r-2)(a-r+1)-(a-1)(a+2-r)}{(a-1)(a+2-r)(a-r+1)} & t^{3} + \frac{(-a^{3}+a^{2}r-2ar+a+2r-2)t}{(a-1)(a-2)(a-r+1)} \end{pmatrix} \\
        & \hspace{1cm} - \begin{pmatrix}(1-r)t^{2}+\frac{r+1-a}{a+1-r} & \frac{2t(1-r)(a+2-r)}{(a-2)(a+1-r)} \\ \frac{2t(r-1)(a-2)}{(a-1)(a+2-r)} & - \frac{t^{2}(ar-a-r+1)+a+1}{a-1} \end{pmatrix}.
    \end{split}
\end{equation}
Hence, using Lemmas \ref{lde} and \ref{conta}, we can decompose each operator $D\in \Upsilon$ in the form
$$
D(y)=\B(\A(y))-y\Psi
$$
where $\B$ and $\Psi$ are given in Lemma \ref{lde}.

Since $P$ is an eigenfunction of $D\in \Upsilon$, we get from Lemma \ref{conta},
that the sequence of matrix polynomials with a gap at $n = 1$ (and non-singular leading coefficient) defined by $\mathcal{P}_{0}(t) = I$ and
$$
\Pp_{n}(t)=\A(P_{n-2}), \quad n\ge 2,
$$
are eigenfunctions of the second-order differential operators
\begin{equation}\label{eq6ls}
\D(y)=\A(\B(y))-yA_1^{-1}\Psi A_1,
\end{equation}
whose coefficients are now matrix rational functions of $t$.
where $\A$ and $\B$ are given in Lemma \ref{lde}.

In particular, for the operator $\tilde D$ (\ref{tdj}), we have
$\tilde D (P(t)) = 0$. It is also easy to check that $P$ satisfies the identity (\ref{lus}). Thus, we can factorize $\tilde D$ as $\tilde D (y)= \mathcal{B}(\mathcal{A}(y))$, with
\begin{equation*}
    \begin{split}
               \mathcal{B} & = \partial \begin{pmatrix} - \frac{(a-1)^{2}(a+2-r)(a+1-r)t}{a^{2}t^{2}-art^{2}+rt^{2}-t^{2}+1} & \frac{(a-1)(a+2-r)(a+1-r)}{a^{2}t^{2}-art^{2}+rt^{2}-t^{2}+1} \\ - \frac{(a-1)(a-2)(a-r+1)}{a^{2}t^{2}-art^{2}+rt^{2}-t^{2}+1} & - \frac{(a-1)(a-2)(a-r+1)^{2}t}{a^{2}t^{2}-art^{2}+rt^{2}-t^{2}+1} \end{pmatrix}.
    \end{split}
\end{equation*}
It is not difficult to see that for $r>0$, $0<a<r$, $a\not =1,2, r-1,r-2$, the conditions
\begin{align*}
&\det (P'(t)^{-1}P(t)U(t))=\frac{(1-t^2)^2(a^2-ar+r-1)t^2+1)}{(1-a)(a-r+1)}\not =0,\quad t\in (-1,1)\\
&\hspace{0cm} \mbox{$F_2W_{a,r}$ is positive definite in $(-1,1)$}
\end{align*}
(where $F_2$ is the coefficient of the second derivative of  $\tilde D$ (\ref{tdj})), are equivalent to
\begin{equation}\label{esbh}
\begin{cases} \mbox{$0<a<1$ and $a<r<a+1$ or $r>a+2$}, &\\ \mbox{$1<a<2$ and $1+a<r<a+2$},& \\ \mbox{$2<a$ and $a<r<a+1$}.&\end{cases}
\end{equation}
Hence, under the assumption (\ref{esbh}), Theorem \ref{t10} gives that the matrix polynomials $\Pp_n$, $n\not = 1$, are orthogonal with respect to the weight matrix
$$
(P U)^{-1}P 'F_2W_{a,r}((P U)^{-1}P ')^*.
$$
A simple computation shows
$$\mathcal{W}_{a,r}(t) = \tfrac{(1-t^{2})^{\frac{r}{2}-2}a(a-1)(a-2)(a+1-r)(a+2-r)(r-a)}{\mathfrak{p}(t)^{2}}Q_{a,r}(t),$$
where $Q_{a,r}$ is the matrix polynomial of degree $4$ given by
$$
{ \left(\begin{smallmatrix}\frac{(a-2)\mathfrak{p}(t)^{2}-(a-r)(a(a-r)+4)\mathfrak{p}(t)+2(a-r)(a^{2}-ar+r)}{(a-2)^2(a+1-r)(r-a)} & \frac{t((r-4)\mathfrak{p}(t)+2(a^{2}-ar+r))}{(a-2)(a+2-r)} \\ \frac{t((r-4)\mathfrak{p}(t)+2(a^{2}-ar+r))}{(a-2)(a+2-r)}  & \frac{(a+2-r)\mathfrak{p}(t)^{2}-a(a^{2}-ar+4)\mathfrak{p}(t)+2a(a(a-r)+r)}{a(a-1)(a+2-r)^{2}} \end{smallmatrix}\right),}
$$
and $\mathfrak{p}(t) = t^{2}(a^{2}-ar+r-1)+1$.

\subsection{Fifth example: constructing examples without using quasi-Darboux transformation}\label{sdel}
We will use Lemma \ref{lede} to construct exceptional matrix polynomials without using quasi-Darboux transformations. The starting point is the example constructed in Subsection \ref{s1}.

Consider the operator $\D$ in (\ref{eq6}) for the parameters
$$
u1=1,\quad  u2=\frac{4} {a^2},\quad u3=0,\quad u4=0,\quad u5=4-\frac{4} {a^2}.
$$
An easy computation gives
\begin{align*}
\D (y)&=y\F_0+y'\F_1+y''\F_2,\\
\F_2(t)&=\frac{1}{\mathfrak p(t)}\begin{pmatrix}
     -a^2(2t^2+1)& -(a^2-2)at(2t^2+1) \\
     \frac{4at}{a^2-2} & 4t^2
    \protect\end{pmatrix},\\
    \F_1(t)&=\frac{1}{\mathfrak p^2(t)}\begin{pmatrix}
     -2t(\mathfrak p^2(t)+4\mathfrak p(t)-4a^2) & \frac{4(a^2-2)t^2}{a}(-\mathfrak p(t)+2a^2) \\
     \frac{4(\mathfrak p(t)+2a^2)(\mathfrak p(t)-a^2)}{a(a^2-2)} & -8a^2t
    \protect\end{pmatrix},\\ \F_0&=\begin{pmatrix}2-\frac{4}{a^2}&0\\0&0\end{pmatrix},
\end{align*}
where $\mathfrak p(t)=\det P_{1,\xi(a)}(t)=2(2-a^2)t^2-a^2\not =0$, $t\in \RR$ (assuming $a^2>2$).

Evaluating at $t=0$, and writing
$$
M=\begin{pmatrix}0&0\\0&1\end{pmatrix},
$$
we get
\begin{equation}\label{eqwb}
\F_2(0) M=0,\quad \F_1(0) M=0,\quad \F_0 M=M\F_0^*.
\end{equation}
For $a^2>2$, consider next the weight matrix $\W_a $ defined in (\ref{wex}), and the exceptional polynomials $\Pp_n$, $n\not =1$ (\ref{pexct1}). Write $\Aa$ for the left linear span of $\Pp_n$, $n\not =1$. We have already proved that the operator $\D$ is symmetric with respect to the pair $(\W,\Aa)$. Hence, it follows easily using (\ref{eqwb}) that the operator $\D$ is also symmetric with respect to the pair $(\W_{a,\zeta},\Aa)$, where $\W_{a,\zeta}$ is the weight matrix
$$
\W_{a,\zeta}=\frac{1}{\sqrt \pi}\W_a+\zeta M\delta_0,\quad \zeta \ge 0
$$
(see also Lemma \ref{lede}). Hence, we can construct a sequence of orthogonal polynomials $\Pp_n^\zeta$, $n=0,2,3,\dots$, which are orthogonal with respect to $\W_{a,\zeta}$ and eigenfunctions of the operator $\D$ (just by applying the Gram-Schmidt orthogonalizing method to $\Pp_n$, $n\not =1$, with respect to $\W_{a,\zeta}$; see Part (2) of Lemma \ref{lsyo}). The first three of these polynomials are:
\begin{align*}
\Pp_0^\zeta&=Id,\\  \Pp_2^\zeta(t)&=\begin{pmatrix}t^2-\frac{a^2+2}{2a^2-4} & -at\\\frac{2at}{a^2-2}&t^2+\frac{1}{2(\zeta(a^2-2)^2+1)}\end{pmatrix}, \\
\Pp_3^\zeta(t)&=\begin{pmatrix}t^3-\frac{3 a^2t}{2a^2-4} & \frac{-3at^2}2 \\\frac{6at^2}{a^4-4}&t^3+\frac{3a^2t}{2a^2+4}\end{pmatrix}.
\end{align*}
Notice that for $\zeta=0$, $\Pp_n^0=\Pp_n$, where $\Pp_n$ are the exceptional polynomials orthogonal with respect to the weight matrix $\W_a$.

Write $\Aa_\zeta$ for the left linear span  of the matrix polynomials $\Pp_{n}^\zeta$, $n=0,2,3,\dots$, in $\RR ^{N\times N}[t]$.
From its construction, we have $\Aa_\zeta=\Aa_0$.  Then the completeness of $\Pp_n^\zeta$, $n=0,2,3,\dots$, in $L^2(\W_{a,\zeta})$, can be proved as in Theorem \ref{th6}.

\bigskip

We do not know if this example can be constructed by applying a Darboux or a quasi-Darboux transformation to a sequence of orthogonal polynomials, but our conjecture is that this could not be the case. Indeed, consider the operator $D$ (\ref{sdop}) for
$$
u1=1,\quad  u2=\frac{4} {a^2},\quad u3=0,\quad u4=0,\quad u5=4-\frac{4} {a^2}.
$$
A simple computation gives
$$
D(y)=y\begin{pmatrix}2-\frac{4}{a^2} & 0\\0&0\end{pmatrix}+y'\begin{pmatrix}-2t & \frac{2}{a}\\-\frac{2}{a}&0\end{pmatrix}+y''\begin{pmatrix}0& at\\0&1\end{pmatrix}.
$$
Consider the first-order differential operator (\ref{eq4}) for $\xi=1-a^2$. Since $P_{1,1-a^2}$ is an eigenfunction of $D$, if we consider any matrix polynomial $P$ which is an eigenfunction of $D$ then
$\A(P)$ is then an eigenfunction of $\D$ (because of Lemma \ref{conta}).
Using Maple we have found the following matrix polynomials $P_{j}^\zeta$, $0\le j\le 3$,
\begin{align*}
P_{0}^\zeta&=-Id,\quad  P_1^\zeta(t)=\begin{pmatrix}t & -\frac{a}2\\\frac{a}{a^2-2}&t\end{pmatrix}, \\   P_2^\zeta(t)&=\begin{pmatrix}t^2-1/2 & -at\\-at+\tau_2&t^2+\frac{\tau_2 (a^2-2)t}{a}+\frac{a^2(\zeta(a^2-2)^2+1)-1}{2(\zeta(a^2-2)^2+1)}\end{pmatrix},\\
P_3^\zeta(t)&=\frac{1}{2}\begin{pmatrix}t^3-3t/2 & -3a(2t^2-1)/4\\\frac{-3at^2}{a^2+2}+\tau_3&t^3+\frac{\tau_3 t(a^2-2)}{a}\end{pmatrix}.
\end{align*}
They satisfy:
\begin{enumerate}
\item The polynomial $P_{j}^\zeta$ is a matrix polynomial of degree $j$ with nonsingular leading coefficient, and it is an eigenfunction of $D$.
\item $\A(P_{1}^\zeta)=0$ and $\Pp_j^\zeta=\A(P_{j}^\zeta)$, $j=2, 3$.
\end{enumerate}
However, except for $\zeta=0$, the polynomials $P_{j}^\zeta$, $1\le j \le 3$, can not be orthogonal with respect to a weight matrix because they do not satisfy a three term recurrence relation for any values of the parameters $\tau_j$.

\section{Recurrence relations for the exceptional matrix polynomials}\label{ssrr}

As we wrote in the Introduction, it is known that exceptional scalar valued polynomials $p_n$, $n\in \NN \setminus X$, do satisfy recurrence relations of the form
$$
q(x)p_n(x)=\sum_{j=-r}^r a_{n,j}p_{n-j}(x),
$$
where $q$ is a certain polynomial of degree $r$ and $(a_{n,j})_{n}$, $j=-r,\cdots ,r$, are sequences of numbers independent of $x$.

Although a discussion on whether the exceptional matrix polynomials satisfy some kind of recurrence relations is beyond the scope of this paper, we just point out that this is the case of all the examples constructed in this paper.

For instance, consider the polynomials $\Pp_{n,\xi}$ defined in Remark \ref{re7m} (let us note that for $\xi=1-a^2$ we get the polynomials studied in Section  \ref{ss1}).
We have checked using Maple and Mathematica that the polynomials $\Pp_{n,\xi}$ satisfy a seven-term recurrence relation of the form
$$
q(t)\Pp_{n,\xi}(t)=\sum_{j=-3}^3 A_{n,j,\xi}\Pp_{n+j,\xi}(t)
$$
where $q$ is the same polynomial for all $\xi$, it has degree $3$ and satisfies $q'(t)=\det(P_{1,1-a^2})$.

This is just the same pattern satisfied by the scalar valued exceptional polynomial. In the case of the polynomials
$\Pp_{2,n}$ (\ref{p 2 step}) constructed in Section \ref{s12}, we have checked that they satisfy a eleven-term recurrence relation of the form
$$
q(t)\Pp_{2,n}(t)=\sum_{j=-5}^5 A_{n,j}\Pp_{2,n+j}(t)
$$
where the polynomial $q$ has degree $5$ and satisfies $q'(t)=\det \Pp_{1,2}$. Let us note that, as in the previous examples, $\det \Pp_{1,2}$ is the polynomial which allows us to write the weight matrix $\W_a$ in the form
$$
\W_a=\frac{e^{-t^2}}{\det ^2 \Pp_{1,2}}R(t),
$$
where $R$ is a matrix polynomial.




\begin{thebibliography}{99}

\bibitem{Atk} F. V. Atkinson, {\em Discrete and continuous boundary problems}, Academic Press, New York, 1964.

\bibitem{Ber} Yu. M. Berezanskii, {\em Expansions in eigenfunctions of selfadjoint operators}, Transl. Math. Monographs, AMS {\bf17} (1968).

\bibitem{Be} N. Bonneux, {\em Exceptional Jacobi polynomials}, J. Approx. Theory {\bf239} (2019), 72--112.

\bibitem{BK} N. Bonneux and A. B. J. Kuijlaars, {\em Exceptional Laguerre polynomials}, Stud. Appl. Math. {\bf141} (2018), 547--595.

\bibitem{BP1} I. Bono Parisi and I. Pacharoni, {\em Darboux transformations and the algebra $\mathcal{D}(W)$}, Linear Algebra Appl. {\bf709} (2025), 203--232.

\bibitem{BP2} I. Bono Parisi and I. Pacharoni, {\em The algebra $\mathcal{D}(W)$ via strong Darboux transformations}, J. Math. Anal. Appl. {\bf549} (2025), No. 1, 22 p., Id/No 129443.

\bibitem{BP3} I. Bono Parisi and I. Pacharoni, {\em Structure of operator algebras for matrix orthogonal polynomials}, Preprint, arXiv:2502.16070 [math.CA] (2025).

\bibitem{CMV} M. J. Cantero, L. Moral, and L. Vel\'azquez, {\em Matrix orthogonal polynomials whose derivatives are also orthogonal}, J. Approx. Theory {\bf146} (2007), 174--211.

\bibitem{CGYZ} W. R. Casper, F. A. Gr\"unbaum, M. Yakimov, and I. Zurri\'an, {\em Matrix valued discrete-continuous functions with the prolate spheroidal property}, Commun. Math. Phys. {\bf405} (2024), No. 3, Paper No. 69, 36 p.

\bibitem{CaGr1} M. M. Castro and F. A. Gr\"unbaum, {\em Orthogonal matrix polynomials satisfying first order differential equations: a collection of instructive examples}, J. Nonlinear Math. Phys. {\bf12} (2005), Suppl. 2, 63--76.

\bibitem{CaGr2} M. M. Castro and F. A. Gr\"unbaum, {\em The algebra of differential operators associated to a given family of matrix valued orthogonal polynomials: five instructive examples}, Int. Math. Res. Not. {\bf2006} (2006), Article ID 47602, 33 p.

\bibitem{ducu} G. Curbera and A. J. Dur\'an, {\em Invariance properties of Wronskian type determinants of classical and classical discrete orthogonal polynomials}, J. Math. Anal. Appl. {\bf474} (2019), 748--764.

\bibitem{DPS} D. Damanik, A. Pushnitski, and B. Simon, {\em The analytic theory of matrix orthogonal polynomials. Quadrature formulas for matrix measures - a geometric approach}, Surv. Approx. Theory {\bf4} (2008), 1--15.

\bibitem{DEK} S. Y. Dubov, V. M. Eleonskii, and N. E. Kulagin, {\em Equidistant spectra of anharmonic oscillators}, Sov. Phys. JETP {\bf75} (1992), 47--53; Chaos {\bf4} (1994), 47--53.

\bibitem{DG} J. J. Duistermaat and F. A. Gr\"unbaum, {\em Differential equations in the spectral parameter}, Commun. Math. Phys. {\bf103} (1986), 177--240.

\bibitem{Dur6} A. J. Dur\'an, {\em Matrix inner product having a matrix symmetric second order differential operator}, Rocky Mount. J. Math. {\bf27} (1997), 585--600.

\bibitem{Durx1} A. J. Dur\'an, {\em Exceptional Charlier and Hermite polynomials}, J. Approx. Theory {\bf182} (2014), 29--58.

\bibitem{Durx2} A. J. Dur\'an, {\em Exceptional Meixner and Laguerre polynomials}, J. Approx. Theory {\bf184} (2014), 176--208.

\bibitem{du00} A. J. Dur\'an, {\em Symmetries for Casorati determinants of classical discrete orthogonal polynomials}, Proc. Amer. Math. Soc. {\bf142} (2014), 915--930.

\bibitem{du01} A. J. Dur\'an, {\em Wronskian type determinants of orthogonal polynomials, Selberg type formulas and constant term identities}, J. Combin. Theory Ser. A {\bf124} (2014), 57--96.

\bibitem{Durx3} A. J. Dur\'an, {\em Exceptional Hahn and Jacobi orthogonal polynomials}, J. Approx. Theory {\bf214} (2017), 9--48.

\bibitem{durr} A. J. Dur\'an, {\em Higher order recurrence relation for exceptional Charlier, Meixner, Hermite and Laguerre orthogonal polynomials}, Integral Transforms Spec. Funct. {\bf26} (2015), 357--376.

\bibitem{Durx4} A. J. Dur\'an, {\em Exceptional Hahn and Jacobi polynomials with an arbitrary number of continuous parameters}, Stud. Appl. Math. {\bf148} (2021), 12451.

\bibitem{durr2} A. J. Dur\'an, {\em The algebra of recurrence relations for exceptional Laguerre and Jacobi polynomials}, Proc. Amer. Math. Soc. {\bf149} (2021), 173--188.

\bibitem{duct} A. J. Dur\'an, {\em Christoffel transform of classical discrete measures and invariance of determinants of classical and classical discrete polynomials}, J. Math. Anal. Appl. {\bf503} (2021), 125306.

\bibitem{duar} A. J. Dur\'an and J. Arves\'u, {\em $q$-Casorati determinants of some $q$-classical orthogonal polynomials}, Proc. Amer. Math. Soc. {\bf144} (2016), 1655--1668.

\bibitem{DuDI1} A. J. Dur\'an and M. D. de la Iglesia, {\em Some examples of orthogonal matrix polynomials satisfying odd order differential equations}, J. Approx. Theory {\bf150} (2008), 153--174.

\bibitem{DuDI2} A. J. Dur\'an and M. D. de la Iglesia, {\em Second order differential operators having several families of orthogonal matrix polynomials as eigenfunctions}, Int. Math. Res. Not. {\bf2008} (2008), Art. ID rnn084.

\bibitem{DuGr1} A. J. Dur\'an and F. A. Gr\"unbaum, {\em Orthogonal matrix polynomials satisfying second order differential equations}, Int. Math. Res. Not. {\bf2004:10} (2004), 461--484.

\bibitem{DuGr3} A. J. Dur\'an and F. A. Gr\"unbaum, {\em Structural formulas for orthogonal matrix polynomials satisfying second order differential equations, I}, Constr. Approx. {\bf22} (2005), 255--271.

\bibitem{DuLo1} A. J. Dur\'an and P. L\'opez-Rodr\'\i guez, {\em The $L^p$ space of a positive definite matrix of measures and density of matrix polynomials in $L^1$}, J. Approx. Theory {\bf90} (1997), 299--318.

\bibitem{FHV} G. Felder, A. D. Hemery, and A. P. Veselov, {\em Zeros of Wronskians of Hermite polynomials and Young diagrams}, Physica D {\bf241} (2012), 2131--2137.

\bibitem{Gan} F. R. Gantmacher, {\em The theory of matrices}, Chelsea Publishing Company, New York, 1960.

\bibitem{GFGM} M. A. Garc\'ia-Ferrero, D. G\'omez-Ullate, and R. Milson, {\em A Bochner type characterization theorem for exceptional orthogonal polynomials}, J. Math. Anal. Appl. {\bf472} (2019), 584--626.

\bibitem{GUGM} D. G\'omez-Ullate, Y. Grandati, and R. Milson, {\em Rational extensions of the quantum harmonic oscillator and exceptional Hermite polynomials}, J. Phys. A: Math. Theor. {\bf47} (2014), 015203.

\bibitem{GUGM2} D. G\'omez-Ullate, Y. Grandati, and R. Milson, {\em Durfee rectangles and pseudo-Wronskian equivalences for Hermite polynomials}, Stud. Appl. Math. {\bf141} (2018), 596--625.

\bibitem{GUKKM} D. G\'omez-Ullate, A. Kasman, A. B. J. Kuijlaars, and R. Milson, {\em Recurrence relations for exceptional Hermite polynomials}, J. Approx. Theory {\bf204} (2016), 1--16.

\bibitem{GUKM1} D. G\'omez-Ullate, N. Kamran, and R. Milson, {\em An extended class of orthogonal polynomials defined by a Sturm-Liouville problem}, J. Math. Anal. Appl. {\bf359} (2009), 352--367.

\bibitem{GUKM2} D. G\'omez-Ullate, N. Kamran, and R. Milson, {\em An extension of Bochner's problem: exceptional invariant subspaces}, J. Approx. Theory {\bf162} (2010), 987--1006.

\bibitem{GPT1} F. A. Gr\"unbaum, I. Pacharoni, and J. Tirao, {\em Matrix valued spherical functions associated to the complex projective plane}, J. Funct. Anal. {\bf188} (2002), 350--441.

\bibitem{GPT2} F. A. Gr\"unbaum, I. Pacharoni, and J. Tirao, {\em An invitation to matrix valued spherical functions: Linearization of products in the case of the complex projective space $P_2({\mathbb C})$}, MSRI Publ. {\bf46} (2003), 147--160.

\bibitem{GPT3} F. A. Gr\"unbaum, I. Pacharoni, and J. Tirao, {\em Matrix valued orthogonal polynomials of the Jacobi type}, Indag. Mathem. {\bf14} (2003), 353--366.

\bibitem{GrTi} F. A. Gr\"unbaum and J. A. Tirao, {\em The algebra of differential operators associated to a weight matrix}, Integr. Equ. Oper. Theory {\bf58} (2007), 449--475.

\bibitem{KMR} E. Koelink, L. Morey, and P. Rom\'an, {\em Matrix exceptional Laguerre polynomials}, Stud. Appl. Math. {\bf152} (2024), 778--809.

\bibitem{Kre1} M. G. Krein, {\em Fundamental aspects of the representation theory of Hermitian operators with deficiency index $(m,m)$}, Ukrain. Math. Zh. {\bf1} (1949), 3--66; Amer. Math. Soc. Transl. (2) {\bf97} (1970), 75--143.

\bibitem{Kre2} M. G. Krein, {\em Infinite J-matrices and a matrix moment problem}, Dokl. Akad. Nauk SSSR {\bf69} (1949), 125--128.

\bibitem{MiTs} H. Miki and S. Tsujimoto, {\em A new recurrence formula for generic exceptional orthogonal polynomials}, J. Math. Phys. {\bf56} (2015), 033502.

\bibitem{OS0} S. Odake and R. Sasaki, {\em Infinitely many shape invariant potentials and new orthogonal polynomials}, Phys. Lett. B {\bf679} (2009), 414--417.

\bibitem{OS} S. Odake and R. Sasaki, {\em Infinitely many shape invariant discrete quantum mechanical systems and new exceptional orthogonal polynomials related to the Wilson and Askey-Wilson polynomials}, Phys. Lett. B {\bf682} (2009), 130--136.

\bibitem{OS1} S. Odake and R. Sasaki, {\em Casorati identities for the Wilson and Askey-Wilson polynomials}, J. Approx. Theory {\bf193} (2015), 184--209.

\bibitem{PZ} I. Pacharoni and I. Zurri\'an, {\em Matrix Gegenbauer polynomials: the $2\times2$ fundamental cases}, Constr. Approx. {\bf43} (2016), 253--271.

\bibitem{Qu} C. Quesne, {\em Exceptional orthogonal polynomials, exactly solvable potentials and supersymmetry}, J. Phys. A {\bf41} (2008), 392001.

\bibitem{Reach} M. Reach, {\em Generating difference equations with the Darboux transformation}, Commun. Math. Phys. {\bf119} (1988), 385--402.

\bibitem{STZ} R. Sasaki, S. Tsujimoto, and A. Zhedanov, {\em Exceptional Laguerre and Jacobi polynomials and the corresponding potentials through Darboux-Crum transformations}, J. Phys. A: Math. Gen. {\bf43} (2010), 315204.

\bibitem{Tir2} J. Tirao, {\em The algebra of differential operators associated to a weight matrix: a first example}, Contemp. Math., Amer. Math. Soc., to appear.

\bibitem{TZ18} J. Tirao and I. Zurri\'an, {\em Reducibility of matrix weights}, Ramanujan J. {\bf45} (2018), 349--374.

\bibitem{Z} I. Zurri\'an, {\em The algebra of differential operators for a matrix weight: an ultraspherical example}, Int. Math. Res. Not. {\bf2017} (2017), No. 8, 2402--2430.

\end{thebibliography}
\end{document}